\documentclass[a4paper,12pt]{article}

\usepackage{amsmath}
\usepackage{amsthm}
\usepackage{amssymb}
\usepackage[mathscr]{eucal}
\usepackage{amscd}
\usepackage{enumerate}
\usepackage[dvips]{color}
\usepackage[dvips]{psfrag,graphicx}
\usepackage{float}

\newtheorem{theorem}{Theorem}[section]

\theoremstyle{definition}

\newtheorem{remark}[theorem]{Remark}

\newcommand{\dx}{{\rm d}x}

\numberwithin{equation}{section}

\title{
The existence and instability of blowing-up steady states for the  Shigesada-Kawasaki-Teramoto competition model with cross-diffusion}

\author{Kousuke Kuto
\\ {\small Department of Applied Mathematics, Waseda University}\\
{\small
3-4-1 Ohkubo, Shinjuku-ku,
Tokyo 169-8555,
Japan}\\
\\
Yaping Wu\footnote{
Corresponding author,
e-mail address:
yaping\_wu@hotmail.com (Y.\,Wu).}
\\ {\small School of Mathematical Sciences, Capital Normal University }\\
{\small
Beijing 100048, China}\\
}

\theoremstyle{plain}
\newtheorem{thm}{Theorem}[section]
\newtheorem{prop}[thm]{Proposition}
\newtheorem{lem}[thm]{Lemma}

\theoremstyle{definition}

\makeatletter

\@addtoreset{equation}{section}

\date{}
\begin{document}
\maketitle
\begin{abstract}
This paper is concerned with the existence and stability of a type of blowing-up positive steady states for a
 shadow system of the Shigesada-Kawasaki-Teramoto two species competition model and the perturbed SKT model  with a large enough  cross diffusion parameter and bounded random diffusion parameters.
The asymptotic behavior of coexistence steady states
to the SKT model as one of the cross-diffusion terms
tends to infinity was studied by Lou and Ni \cite{LN2}
and it was revealed that
almost all coexistence states can be characterized by
one of three  shadow systems.
In \cite{Ku1},
the first author of this paper studied one of these
shadow systems with one cross diffusion term in one space dimension
and proved that a component on the bifurcation branch
blows up
as a bifurcation parameter approaches the least positive eigenvalue of $-\Delta$
with homogeneous Neumann boundary condition.
By applying a different approach from \cite{Ku1} based on special transformation and Lyapunov-Schmidt reduction method, in this paper we derive the existence and the detailed asymptotic structure of several branches of positive steady states
near the blow-up points  in one or multi-dimensional cases for the shadow system with one or two cross diffusion terms, and all the
branches of the large steady states obtained for the shadow system near the blow-up points are proved to be spectrally unstable.
 Further by applying perturbation argument, we also proved  the existence and instability  of several branches of
 the perturbed  positive steady states for the original SKT model  when  one of cross diffusion parameter is large enough.

\medskip
%{\it 2000 Mathematics Subject Classification}:
\end{abstract}
%%%%%%%%%%%%%%%%%%%%%%%%%%%%%%%%%%%%%%%%%%%%%%%%%%%%%%%%%%%%%%%%%%%%%%%%%%

\section{Introduction}
In this paper,
we are concerned with the following
Lotka-Volterra competition model
with one or two quasilinear cross diffusion terms
\begin{equation}\label{SKT}
\begin{cases}
u_{t}=d_1\Delta[\,(1+\alpha v)u\,]+u(a_{1}-b_{1}u-c_{1}v)\ \ &\mbox{in}\
\Omega\times (0,T),\\
v_{t}=d_2\Delta[\,(1+\beta u)v\,]+v(a_{2}-b_{2}u-c_{2}v)\ \ &\mbox{in}\
\Omega\times (0,T),\\
\frac{\partial u}{\partial \nu }=\frac{\partial v}{\partial \nu }=0\ \ &\mbox{on}\
\partial\Omega\times (0,T),\\
u(x,0)=u_{0}(x)\ge 0,\ v(x,0)=v_{0}(x)\ge 0
\ \ &\mbox{in}\ \Omega,
\end{cases}
\end{equation}
where $0<T\le +\infty$; $\Omega\,(\subset\mathbb{R}^{N})$ is a bounded domain
with a smooth boundary $\partial\Omega$; $\nu(x)$ denotes the  outward unit  normal vector  at $x\in\partial\Omega$; the unknown functions $u(x,t)$ and $v(x,t)$,
respectively, stand for the population densities of
competing species
at location $x$ in the habitat $\Omega$
and time $t>0$; and the
positive constants
$(a_{1}, a_{2})$,
$(b_{1}, c_{2})$
and
$(c_{1}, b_{2})$
represent the birth rates of respective species,
degrees of the intra-specific competitions
and
the degrees  of the inter-specific competitions; and the
positive constants $d_{i}$
represent the rates of diffusion caused by
random movement of individuals of each species.
The nonlinear diffusion term
$d_1\alpha \Delta (uv)$
(resp. $d_2\beta \Delta (uv)$)
is usually referred as
the {\it SKT type of  cross-diffusion} term which describes an ability
of the species of $u$ (resp. $v$) which diffuses from the high
density region of the competitor $v$ (resp. $u$)
toward the low density region;
we refer to \cite{OL} for the ecological background of
such cross-diffusion terms.

A class of Lotka-Volterra competition models
with the SKT type of quasilinear cross-diffusion terms were first proposed
by Shigesada, Kawasaki and Teramoto \cite{SKT}
for the purpose of realizing the segregation phenomena
of two competing species observed
in the real ecosystem, and the
system \eqref{SKT} is a simplified  Shigesada-Kawasaki-Teramoto competition model with cross diffusion under the homogeneous Neumann boundary condition,  which will be simply called the SKT  system in the following of this paper.

 Since the pioneering work \cite{SKT}, the effect of  cross-diffusion on  global existence of solutions in time, and the existence and stability/instability of nontrivial steady states for SKT types of cross diffusion systems have been widely and deeply investigated by many mathematicians, and the review of  some  earlier important research works on SKT types of cross diffusion systems can be found in the survey papers \cite{Ju,Ni,Yam1,Yam2}.

In this paper we  focus on the investigation of the  existence and  stability/instability of some special types of co-existence  steady states $(u(x),v(x))$ to the SKT system \eqref{SKT} with a large enough cross diffusion parameter,  where the steady states are nontrivial positive  solutions
to the Neumann boundary value  problem of the following nonlinear  elliptic system:
\begin{equation}\label{sSKT}
\begin{cases}
d_1\Delta[\,(1+\alpha v)u\,]+u(a_{1}-b_{1}u-c_{1}v)=0\ \ &\mbox{in}\
\Omega,\\
d_2\Delta[\,(1+\beta u)v\,]+v(a_{2}-b_{2}u-c_{2}v)=0\ \ &\mbox{in}\
\Omega,\\
\frac{\partial u}{\partial \nu }=\frac{\partial v}{\partial \nu }=0\ \ &\mbox{on}\
\partial\Omega,\\
u> 0,\ v> 0
\ \ &\mbox{in}\ \Omega.
\end{cases}
\end{equation}
The first important theoretical work on the existence of nontrivial positive steady states of \eqref{SKT} with a large  cross diffusion parameter is due to Mimura, Nishiura, et al.\cite{MNTT}, in which it was shown that for the strong competition case $\frac{b_1}{b_2}<\frac{c_1}{c_2}$ when $d_2>0$ is small,  $d_1$ is large enough and $\alpha>0$ is not small, the SKT model  \eqref{SKT} with one dimensional $\Omega$ admits  several types of positive steady states with transition layers or boundary layers; and the spectral stability/instability of such steady states of \eqref{SKT} were proved in \cite{Ka1}.

The existence/nonexistence of nontrivial positive solutions   to the stationary  SKT  system \eqref{sSKT} with one dimensional or multi-dimensional $\Omega$  were widely and deeply investigated in \cite{LN1} and \cite{LN2}. Especially for the system \eqref{sSKT} with large $\alpha $ and bounded $d_2>0$ and small $\beta\ge 0$, Lou and Ni \cite{LN2} proved the uniform boundedness of all positive steady  states and
found several types of  {\it shadow systems} which can characterize
the limits of positive solutions of \eqref{sSKT} as
$\alpha\to\infty$ with $d_1\to d_\infty>0$ or $d_1\to \infty$
for the fixed $d_2$ and $\beta\ge 0$.
To be precise,
three types of  shadow systems were obtained in the following sense:

\begin{thm}[{\cite[Theorem 1.4]{LN2}}]\label{LNthm}
Suppose that $N\le 3$,
$a_{1}/a_{2}\neq b_{1}/b_{2}$,
$a_{1}/a_{2}\neq c_{1}/c_{2}$
and
$a_{2}/d_{2}\neq\lambda_{j}$ for
any positive eigenvalue $\lambda_{j}$
of $-\Delta$
with the homogeneous Neumann boundary condition on $\partial\Omega$.

Let $(u_n, v_n)$ be a sequence of positive nontrivial  solutions of \eqref{sSKT} with $(d_1,\;\alpha)=(d_{1,n},\alpha_n)$ and each fixed $d_2>0$ and $\beta\ge 0$, then there exists a small
$\delta=$ $\delta (a_{i}, b_{i}, c_{i}, d_{i})>0$ such that
if $\beta\le\delta$,
the following conclusions hold.

(1) If $\alpha_n\rightarrow\infty$ and $d_{1,\;n}\rightarrow d_1\in(0,\;\infty)$ as $n\to +\infty$, then by passing to a  subsequence if necessary, the sequence of $\{(u_n, v_n)\}$ must satisfy either (i) or (ii) below;

(2) If $\alpha_{n}\rightarrow\infty$ and $d_{1,\;n}\rightarrow \infty$ as $n\to +\infty$, then by passing to a subsequence if necessary, the sequence $\{(u_n, v_n)\}$ must satisfy either (i) or (iii) below, where

\noindent (i) $(u_n,\; v_n)$ converges uniformly to positive $(\frac{\tau}{v}, v)$ as $n\to +\infty$, where $\tau $ is a positive constant  and $(\tau, v(x) )$ satisfies
\begin{eqnarray}\label{LNlim1}
\begin{cases}
\displaystyle\int_{\Omega}
\dfrac{1}{v}\left(
a_{1}-\dfrac{b_{1}\tau }{v}-c_{1}v\right) =0\; ,\\
d_{2}\Delta v+v(a_{2}-c_{2}v)-b_{2}\tau =0
\ \ &\mbox{in}\ \Omega,\\
\frac{\partial v}{\partial \nu }=0
\ \ &\mbox{on}\ \partial\Omega;\\
\end{cases}
\end{eqnarray}
(ii)   $(u_n,\; \alpha_n v_n)$ converges uniformly to positive $(u, w)$ as $n\to +\infty$, where $(u, w)$ is  a positive solution of
\begin{eqnarray}\label{LNlim2}
\left\{
 \begin{array}{ll}
 d_1\Delta[(1 +w)u]+u(a_1-b_1u)=0 \;\; & \mbox{in}\ \Omega,\\
 d_2\Delta [(1+\beta u)w]+w(a_2-b_2u)=0 \;\; &\mbox{in}\ \Omega,\\
 \frac{\partial u}{\partial \nu }=
\frac{\partial w}{\partial \nu }=0 \;\; &\mbox{on}\ \partial\Omega;
\end{array}
\right.
\end{eqnarray}
(iii) $(u_n,\;\alpha_n v_n)$ converges uniformly to positive $(\frac{\tau}{1+w}, w)$, where
$\tau$ is a positive constant and $(\tau, w(x))$ satisfies
\begin{eqnarray}\label{LNlim3}
\left\{
 \begin{aligned}
&\int_{\Omega}\frac{a_1}{(1+w)}-\int_{\Omega}\frac{b_1\tau}{(1+w)^2}=0,\\
&d_2\Delta [(1+\beta \frac{\tau}{1+w})w]+w\left(a_2-\frac{b_2\tau}{1+w}\right)=0\   \mbox{in}\ \Omega,\\
&w>0,\ x\in \Omega;\ \frac{\partial w}{\partial \nu}=0\;\; \mbox{on} \ \partial\Omega.
\end{aligned}
\right.
\end{eqnarray}
\end{thm}

Theorem \ref{LNthm} gives
three shadow systems
\eqref{LNlim1}, \eqref{LNlim2} and \eqref{LNlim3} as $\alpha\to\infty$.
In this paper, we call \eqref{LNlim1}
{\it the 1st shadow system}, \eqref{LNlim2}
{\it the 2nd shadow system}, and \eqref{LNlim3}
{\it the 3rd shadow system}.
It can be ecologically said that
the first shadow system \eqref{LNlim1}
characterizes the segregation in the sense that
$v$ has a positive limit and  $uv\to\tau$ as $\alpha\to +\infty$ with bounded $d_1$  or $d_1\to \infty$,
whereas
the 2nd  system \eqref{LNlim2} characterizes the extinction
of the species of $v$ with the order $1/\alpha$ as $\alpha\to +\infty$ with bounded or fixed $d_1$, and 3rd shadow system \eqref{LNlim2} characterizes the extinction
of the species of $v$ as $\alpha\to +\infty$ and $d_1\to +\infty$.  In  this paper we always denote $A=\frac{a_1}{a_2}$, $B=\frac{b_1}{b_2}$ and $C=\frac{c_1}{c_2}$.

Since the pioneering work of Lou and Ni \cite{LN2}, the  1st shadow system \eqref{LNlim1} has been deeply and widely investigated by some mathematicians (see \cite{KW,LNY1,LNY2,MSY,Ni,Wu,WX,Yo}). In the one dimensional case with $\Omega=(0,1)$, Lou, Ni and Yotsutani \cite{LNY1} obtained  remarkable results in which
almost all solutions of \eqref{LNlim1}
are explicitly expressed by elliptic functions.
In \cite{LNY1}, the diffusion coefficient $d_{2}$ is
employed as a main parameter to analyze the set
of solutions.
As a rough description of results obtained in \cite{LNY1} for the 1st shadow system \eqref{LNlim1},
for the case when  $A>B$,
the existence range of $d_{2}$ for the  positive strictly monotone solutions
is contained in $(0, a_{2}/\pi^2)$.
Further, in \cite{LNY1} they obtained the asymptotic behavior
of the monotone solutions for the 1st shadow system \eqref{LNlim1} as $d_{2}\to 0$ or $d_{2}\to a_{2}/\pi^2$.
In \cite{LNY1} for the case $A>\frac{1}{4}B+\frac{3}{4}C\;\;{\rm if}\; B<C$, and the case $A>\frac{B+C}{2}\;\; {\rm if}\; B>C$; it was proved that for small enough $d_2>0$ there exist  positive spiky steady states  to the 1st shadow system \eqref{LNlim1} with $\Omega=(0,1)$, which are naturally expected to be perturbed to a type of large spiky  steady states of the original SKT  system \eqref{sSKT} with large $\alpha$ and small $d_2$.

In  \cite{WX} for any  $A>\frac{B+3C}{4}$ and small $d_2>0$,  the second author of the present paper and her cooperator obtained the existence and detailed spiky structure of large spiky steady states to the original SKT system \eqref{sSKT} with $\Omega=(0,1)$ when both $d_1$ and $\alpha$ are large enough.  Kolokolnikov and Wei \cite{KW} obtained the existence and the stability/instability of  large spiky steady states in some multi-dimensional case of  some simplified SKT competition model \eqref{SKT} with large cross diffusion.

In \cite{LNY1} it was also proved that for any $A>B$ the 1st shadow system \eqref{LNlim1} with $\Omega=(0,1)$ has a type of  nontrivial positive steady states $(\xi_{d_2}, \psi_{d_2}(x))$ with singular bifurcation structure when $d_2$ is less than and near $a_2/\pi^2$, where $(\xi_{d_2}, \psi_{d_2}(x))\to (0,0)$ and $\frac{\xi_{d_2}}{\psi_{d_2}(x)}\to $  $\frac{a_2}{b_2}\frac{1}{1-\sqrt{1-B/A}\cos(\pi x)}$ as $d_2\to a_2/\pi^2$. Ni, Wu and Xu  \cite{NWX} obtained the detailed structure of  $(\xi_{d_2}, \psi_{d_2}(x))$ as  $d_2\to a_2/\pi^2$ and proved the existence and stability of  steady states $(u(x),v(x))$ perturbed from $(\xi_{d_2}/\psi_{d_2}(x), \psi_{d_2}(x))$ for the original SKT model   \eqref{SKT} with $\Omega=(0,1)$ and $\beta=0$ when both $d_1$  and  $\alpha$ are large enough.  For some multi-dimensional $\Omega$, by applying a different approach from \cite{NWX}, Lou, Ni and Yotsutani \cite{LNY2} proved  the  existence and  stability of several branches of
 positive steady states for the corresponding evolution problem of the 1st shadow system \eqref{LNlim1} and for the original SKT model \eqref{SKT} with large $d_1$ and  $\alpha$ when $d_{2}$ is slightly less than $a_{2}/\lambda_{j}$. Recently, a series of works were obtained by Yotsunani's group
(\cite{ MSY, Yo})
for further results on the 1st shadow system \eqref{LNlim1} with one dimensional $\Omega$.

As far as we know, there are not much work  on the  shadow  system \eqref{LNlim2} or \eqref{LNlim3}, where the $v$ component  of the steady state to the original SKT system \eqref{sSKT}  is assumed to  tend to zero but  $\alpha v(x)$  tends to a positive and bounded function as $\alpha\to +\infty$. The first theoretical work based on the investigation of the 3rd shadow system \eqref{LNlim3} (the limiting system of 2nd shadow system \eqref{LNlim2} as $d_1\to +\infty$) is due to Lou and Ni \cite{LN2}, in which for any $A>B$ and small enough $d_2>0$ and when
$\Omega$ is an interval, they proved the existence of positive spiky steady states
 to the 3rd shadow system \eqref{LNlim3} and the existence of the corresponding spiky steady states  near $(a_1/b_1,0)$ to the original SKT model \eqref{SKT} with $\beta=0$ and small $d_2>0$ when both $d_1$ and $\alpha$ are large enough; such  spiky steady states of  the evolution 3rd shadow system \eqref{LNlim3} and the original SKT model \eqref{SKT} with large $d_1$ and $\alpha$ are proved to be unstable in \cite{WWX}.

However, there was no paper discussing
the 2nd shadow system \eqref{LNlim2} (with bounded $d_1>0$) before \cite{Ku1}.
In \cite{Ku1}, the first author of the present paper
studied the  bifurcation
structure of the set of nontrivial positive steady states
to the 2nd shadow system \eqref{LNlim2}
with $\beta =0$; in which  it was shown that if $d_{2}\lambda_{1}<a_{1}b_{2}/b_{1}$,
then monotone positive solutions bifurcate from a positive constant
solution at $a_{2}=a_{1}b_{2}/b_{1}$ in the multi-dimensional case.
Furthermore, Li and the second author of the present paper \cite{LW1}
 derived more detailed structure and the instability of the local bifurcating steady states
near the bifurcation point.

For the 2nd shadow system \eqref{LNlim2} with  $\beta=0$ and 1-dimensional $\Omega$,
 it was also proved
in \cite{Ku1} that there exists a global branch
of positive steady state $(u(x),w(x))$ for
$d_{2}\pi^{2}<a_{2}<a_{1}b_{2}/b_{1}$
bifurcating from $a_{2}=a_{1}b_{2}/b_{1}$
and  blowing up as $a_{2}$ tends to $d_{2}\lambda_{1}$.
By replacing the bifurcation parameter $a_{2}$ by $d_{2}$,
the global bifurcating  result can be interpreted as that
the $w$ component on the branch blows up
as $d_{2}$ approaches $a_{2}/\lambda_{1}$.
Furthermore, in \cite{Ku1} it was verified that
the asymptotic profiles
of solutions of the two shadow systems
\eqref{LNlim1} and \eqref{LNlim2}
as $d_{2}\to a_{2}/\lambda_{1}$
coincide with each other after normalization
in the case when $\beta =0$ and $\Omega$ is an interval. It can be conjectured that  for the multi-dimensional case the 2nd shadow system \eqref{LNlim2} still admits a positive  solution branch with $w$-component blowing up as $d_2$ tends to  $a_{2}/\lambda_{1}$, and it is naturally guessed   that the original SKT model \eqref{sSKT}
with sufficiently large $\alpha$ and $\beta=0$ admits some  lower  branches of positive solutions perturbed from some
solution branches of the 2nd
shadow system \eqref{LNlim2}, which connect some upper branches  of positive solutions perturbed from the
solution branches of the 1st
shadow system \eqref{LNlim1} as $d_2\uparrow a_{2}/\lambda_{1}$.

Recently Li and the second author of present paper \cite{LW2}(for the case $\beta=0$) and Tang \cite{Tang} (for the case $\beta>0$) proved the existence and instability of several branches of  positive steady states   bifurcating from infinity to the 3rd shadow system \eqref{LNlim3}  when $d_2$ is near $a_2/\lambda_1$ and $\Omega$ can be multi-dimensional,  by applying special transformation and detailed  spectral analysis  combined with Lyapunov-Schmidt decomposition technique. The results obtained in \cite{LW2} and \cite{Tang} also verify the above  conjecture about the existence of some lower branches of positive steady states to the original SKT model when both $d_1$ and $\alpha$ are large enough and $d_2$ is less than and near $a_2/\lambda_1$.  However the transformation and decomposition techniques  applied in \cite{LW2} and \cite{Tang} cannot be applied  to the 2nd shadow system (when $d_1$ is bounded and fixed) directly, and  the asymptotic structures of  the branches of blow-up steady states to the 2nd shadow system seem to be more  complicated than those of the 3rd shadow system, which will be stated and remarked in  the following section.

The asymptotic behavior of positive steady states  to the  SKT system \eqref{sSKT}
 when one or two cross diffusion parameters tend to infinity have been pursued in various settings
by several research groups in recent years. In \cite{KY1, KY2}, the first author of the present paper and Yamada investigated the asymptotic behavior of nontrivial positive steady states to the  SKT system \eqref{sSKT} as $\alpha \to \infty$ with  \(d_{1}\) kept bounded under Dirichlet boundary conditions, and proved that the 2nd shadow system \eqref{LNlim2} with Dirichlet boundary condition is  the unique limiting system, the global bifurcation structures of the positive steady states to the related limiting system were also established in \cite{KY1, KY2}.

The first author of the present paper and his group also studied the asymptotic behavior of  nontrivial positive steady states of \eqref{sSKT} when both  $\alpha$ and $\beta$ tend to infinity with the same rate (the so-called full cross-diffusion limit).  Under the Neumann boundary conditions as in \eqref{sSKT},
the existence of  a  unique limiting system for positive steady states was established in \cite{Ku2},
and the global bifurcation structure of positive steady states  in one dimensional case was obtained in \cite{Ku3}.
On the other hand, for the stationary SKT problem \eqref{sSKT} with  Dirichlet boundary conditions, two distinct limiting systems arise under the full cross-diffusion limit, and their global bifurcation structures
were investigated in \cite{IKS}.
Furthermore, the stability of the global bifurcation branches obtained therein
was analyzed in \cite{KS}.

It is worth mentioning that due to some essential or technical  difficulties  in proving the  uniform boundedness of some   norms of  the classical solutions with time, such as $\|u_t(t,\cdot)\|_{L_\infty(\Omega)}$ or $\|\alpha v(t,\cdot)\|_{L_\infty(\Omega)}$, to the evolution SKT model \eqref{SKT} with large enough $\alpha$, even in  finite time interval $[0,T]$ and in  one dimensional $\Omega$; or in proving the  global existence of solution in time for any given nonnegative initial data in higher dimensional $\Omega$ or when both $\alpha$ and $\beta$ are positive; it seems that the investigation on all the possible  limiting systems or the asymptotic behavior of all the positive solutions  with time to the  evolution SKT model \eqref{SKT} under large  cross-diffusion limit (as $\alpha\to +\infty$) are more difficult or more complicated than those of the  stationary SKT problems. However under some additional assumptions on the uniform boundedness  or  special types of boundedness of solutions with time, several types of  evolution shadow  systems of SKT model \eqref{SKT} under large cross diffusion limit can be similarly derived whose steady states satisfy   the related stationary shadow systems (see also Kan-on~\cite{Ka3} for the related argument).  The investigation on the spectral  stability/instability  on the steady states to the corresponding evolution shadow system usually play key roles in proving  the spectral stability/instability of the perturbed steady states for the original evolution SKT  system \eqref{SKT} with large cross diffusion.  The existence of  new types of   stable  nontrivial positive steady states to the original SKT models with large cross diffusion also indicates the appearance of new types of pattern formation induced by large cross diffusion, in this sense the stability analysis of  nontrivial steady states to the  SKT model is also  interesting or more important   in both math and application.

In this paper we shall apply a different approach from \cite{Ku1,LW2,Tang} to investigate  the existence and the detailed asymptotic behavior  of  several  branches of  positive steady states bifurcating from infinity  to the  more general 2nd shadow system \eqref{LNlim2} with bounded $\beta\ge 0$ and fixed $d_1>0$ when  $d_{2}$ is near $a_{2}/\lambda_{j}$, where $\Omega$ can be  multi-dimensional, our new results  also verify the above mentioned  conjectures when $d_{2}$ is close to $a_{2}/\lambda_{j}$. In this paper we also study  the spectral stability/instability  of large steady states  to the evolution 2nd shadow system \eqref{LNlim2} with $\beta\ge 0$  and $d_1>0$ when $d_{2}$ is near $a_{2}/\lambda_{j}$. Furthermore we also investigate  the existence and the stability  of the perturbed steady states to the original SKT system \eqref{SKT} with bounded $\beta\ge 0$  and fixed $d_1>0$ when $d_{2}$ is near $a_{2}/\lambda_{j}$ and $\alpha$ is large enough. The results obtained in this paper also cover the case when  $d_1$ is large enough.

\section{Formulation of the problem and statement of main results}

For convenience of later investigation on the detailed asymptotic behavior and stability/instability of the blowing-up steady states to the 2nd
shadow system or to the original system \eqref{SKT} when $d_{2}$ is close to $a_{2}/\lambda_{j}$, we first investigate  the original evolution system \eqref{SKT} and the corresponding evolution system of the 2nd shadow system after a series of  transformations.

  Let $(u_n(x,t),v_n(x,t))$ be a sequence of positive solutions of the original SKT evolution system \eqref{SKT} with cross diffusion parameters $\alpha=\alpha_n$ and fixed $\beta\ge 0$ and  smooth nonnegative initial data $(u_0(x),v_{n0}(x))$, after the transformation $w=\alpha_n v$, then $(u_n(x,t),w_n(x,t))$ satisfy   the following initial boundary value problem:
\begin{equation}\label{2.2}
\begin{cases}
u_{nt}=d_{1}\Delta[\,(1+w_n)u_n\,]+u_n(a_{1}-b_{1}u_n-c_{1}\frac{w_n}{\alpha_n})\ \ &\mbox{in}\
\Omega\times (0,T),\\
w_{nt}=d_{2}\Delta[\,(1+\beta u_n)w_n\,]+w_n(a_{2}-b_{2}u_n-c_{2}\frac{w_n}{\alpha_n})\ \ &\mbox{in}\
\Omega\times (0,T),\\
\frac{\partial u_n}{\partial \nu}=\frac{\partial w_n}{\partial \nu}=0\ \ &\mbox{on}\
\partial\Omega\times (0,T),\\
u_n(x,0)=u_0(x)\ge 0,\;w_n(x,0)=\alpha_n v_{n0}(x)\ge 0 &\mbox{in}\
\Omega.\\
\end{cases}
\end{equation}
Assume that $\alpha_n\to +\infty$ and $\alpha_n v_{n0}(x)\to w_0(x)\ge 0$ in $W^{1,p}(\Omega)$ ($p>N$) as $n\to +\infty$ and    $\|w_n(\cdot,t)\|_{W^{1,p}(\Omega)}$ and  $\|u_n(\cdot,t)\|_{W^{1,p}(\Omega)}$ are  bounded as $\alpha_n \rightarrow\infty$ for the bounded $\beta\geq 0$ and any $t\in [0,T)$,
then it can be proved that $(u_n(x,t),w_n(x,t))\to (u(x,t),w(x,t))$ in some classical sense for $t\in [0,T)$, where  $(u,w)$ is the unique classical nonnegative solution to   the following limiting  system of  (\ref{2.2}) as $\alpha\to +\infty$:
\begin{equation}\label{2.2*}
\begin{cases}
u_{t}=d_{1}\Delta[\,(1+w)u\,]+u(a_{1}-b_{1}u)\ \ &\mbox{in}\
\Omega\times (0,T),\\
w_{t}=d_{2}\Delta[\,(1+\beta u)w\,]+w(a_{2}-b_{2}u)\ \ &\mbox{in}\
\Omega\times (0,T),\\
\frac{\partial u}{\partial \nu}=\frac{\partial w}{\partial \nu}=0\
\ \ &\mbox{on}\
\partial\Omega\times (0,T),\\
u(x,0)=u_0(x)\ge 0,\;w(x,0)=w_0(x)\ge 0&\mbox{in}\
\Omega.\\
\end{cases}
\end{equation}
In this paper, we let $\{\lambda_j\}_{j=0}^\infty$ denote the eigenvalues of $-\Delta$
with homogeneous Neumann boundary condition on $\partial\Omega$, ordered increasingly:
\[
0=\lambda_0 < \lambda_1 < \lambda_2 < \cdots.\]

Assume that $a_{2}\lambda_{j}^{-1}-d_{2}>0$ is sufficiently small and denote
$\varepsilon = a_{2}\lambda_{j}^{-1}-d_{2}>0$ for some simple eigenvalue $\lambda_{j}$ ($j\geq 1$).
We intend to prove that, under this assumption the shadow system \eqref{2.2*} admits a branch of blow-up steady states
$(u(x,d_{2}),w(x,d_{2}))$ for $N\geq 1$, $\beta\geq 0$, with
\[
\|w(x,d_{2})\|_{L^{\infty}(\Omega)} \to +\infty,
\qquad
\|u(x,d_{2})\|_{L^{\infty}(\Omega)} \ \text{bounded,\; as}\ d_{2}\uparrow a_{2}/\lambda_j.
\]

Denote $\varepsilon=a_{2}\lambda^{-1}_{j}-d_{2}$,
after the transformation
\begin{equation}\label{change}
\tilde{w}=\varepsilon w,~~\phi=\varepsilon(1+w)u=(\varepsilon+\tilde{w})u,~
\psi=(1+\beta u)\tilde{w}
\end{equation}
for $ u,w>0$,
then the evolutional system \eqref{2.2*} becomes
\begin{equation}\label{2.4}
\begin{cases}
\varepsilon u_{t}=
d_{1}\Delta[(\varepsilon+\tilde{w})u]+\varepsilon u(a_{1}-b_{1}u) \ \ &\mbox{in}\
\Omega\times (0,T),\\
\tilde{w}_{t}=d_{2}\Delta[(1+\beta u)\tilde{w}]+
\tilde{w}(a_{2}-b_{2}u)\ \ &\mbox{in}\
\Omega\times (0,T),\\
\frac{\partial u}{\partial \nu}=\frac{\partial \tilde{w}}{\partial \nu}=0\
\ \ &\mbox{on}\
\partial\Omega\times (0,T),\\
\end{cases}
\end{equation}
Using
$$
u=\frac{\phi}{\varepsilon+\tilde{w}},~~
\tilde{w}=\frac{\psi}{1+\beta u}
=\frac{(\varepsilon+\tilde{w})\psi}{\beta\phi+\varepsilon+\tilde{w}}
$$
and $\psi>\beta\phi$ for $\beta>0$ and $u,\tilde{w}>0$,
then for any $\beta\ge 0$
\begin{equation}\label{2.5}
\begin{cases}
u=h_{1\varepsilon}(\phi,\psi):=
\frac{2\phi}{\psi-\beta\phi+\varepsilon
+\sqrt{(\psi-\beta\phi+\varepsilon)^{2}+4\varepsilon\beta\phi}},\\[2mm]
\tilde{w}=h_{2\varepsilon}(\phi,\psi)
:=\frac{\psi-\beta\phi-\varepsilon
+\sqrt{(\psi-\beta\phi+\varepsilon)^{2}+4\varepsilon\beta\phi}}{2}.\\
\end{cases}
\end{equation}
Thus the system (\ref{2.4}) can be written as the evolutional system of $(\phi,\psi)$:
\begin{equation}\label{2.6}
\begin{cases}
\varepsilon (h_{1\varepsilon}(\phi,\psi))_{t}=
d_1\Delta\phi+\varepsilon h_{1\varepsilon}(a_{1}-b_{1}h_{1\varepsilon})\ \ &\mbox{in}\
\Omega\times (0,T),\\
(h_{2\varepsilon}(\phi,\psi))_{t}=d_{2}\Delta\psi+
h_{2\varepsilon}(a_{2}-b_{2}h_{1\varepsilon})\ \ &\mbox{in}\
\Omega\times (0,T),\\
\frac{\partial \phi}{\partial \nu}=\frac{\partial \psi}{\partial \nu}=0\
\ \ &\mbox{on}\
\partial\Omega\times (0,T).\\
\end{cases}
\end{equation}
Using $\varepsilon=\frac{a_{2}}{\lambda_j}-d_{2}$,
$\hat{\varepsilon}=\frac{1}{d_{2}}-\frac{\lambda_{j}}{a_{2}}\sim
\frac{\lambda^{2}_{j}}{a^{2}_{2}}\varepsilon$ as $\varepsilon\downarrow0$;
and note that
\begin{equation}\label{2.8}
\begin{cases}
h_{1\varepsilon}(\phi,\psi)=h_{10}(\phi,\psi)
+\varepsilon h^*_1(\phi,\psi)+O(\varepsilon^{2}),\\
h_{2\varepsilon}(\phi,\psi)=
h_{20}(\phi,\psi)
+\varepsilon h^*_{2}(\phi,\psi)+O(\varepsilon^{2}),
\end{cases}
\end{equation}
with
\begin{equation}\label{h}
\begin{cases}
h_{10}(\phi,\psi)=\frac{\phi}{\psi-\beta\phi},~~~~~~
h^*_1(\phi,\psi)=\frac{\partial h_{1\varepsilon}}{\partial \varepsilon}
(\phi,\psi)|_{\varepsilon=0}=-\frac{\phi}{(\psi-\beta\phi)^{2}}
[1+\frac{\beta\phi}{\psi-\beta\phi}],\\
h_{20}(\phi,\psi)=\psi-\beta\phi,~~~
h^*_{2}(\phi,\psi)=\frac{\partial h_{2\varepsilon}}{\partial \varepsilon}
(\phi,\psi)|_{\varepsilon=0}=
\frac{\beta\phi}{\psi-\beta\phi},
\end{cases}
\end{equation}
then the boundary value problem of  (\ref{2.6}) can be rewritten as
\begin{equation}\label{1}
\begin{cases}
\frac{\varepsilon}{d_1} (h_{1\varepsilon}(\phi,\psi))_{t}=\Delta\phi+\varepsilon f_{1}(\phi,\psi,\varepsilon)\ \
&\mbox{in}\;
\Omega\times (0,T),\\[2mm]
\frac{1}{d_2} (h_{2\varepsilon}(\phi,\psi))_{t}=\Delta\psi+\lambda_{j}\psi
-\left(\beta +\frac{b_2}{a_2}\right)\lambda_j\phi
+\varepsilon f_{2}(\phi,\psi,\varepsilon)\ \ &\mbox{in}\;
\Omega\times (0,T),\\[2mm]
\frac{\partial \phi}{\partial \nu}=\frac{\partial \psi}{\partial \nu}=0 \ \ &\mbox{on}\
\partial\Omega\times (0,T),\\
\end{cases}
\end{equation}
with
\begin{equation}\label{3}
\begin{cases}
f_{1}(\phi,\psi,\varepsilon)=\frac{h_{1\varepsilon}}{d_1}(a_{1}-b_{1}h_{1\varepsilon}),\\
f_{2}(\phi,\psi,\varepsilon)=\frac{1}{\varepsilon}\left[\frac{h_{2\varepsilon}}{d_2}
(a_{2}-b_{2}h_{1\varepsilon})-\frac{h_{20}\lambda_j}{a_2}(a_{2}-b_{2}h_{10})\right],
\end{cases}
\end{equation}
and it is easy to check that as $\varepsilon\to 0$,
\begin{equation}\label{4}
\begin{array}{l}f_{1}(\phi,\psi,\varepsilon)\rightarrow
f^0_{1}(\phi,\psi):=\frac{h_{10}}{d_1}(a_{1}-b_{1}h_{10})
=\frac{\phi}{d_{1}(\psi-\beta\phi)}
\left(a_{1}-b_{1}\frac{\phi}{\psi-\beta\phi}\right),\\[2mm]
f_{2}(\phi,\psi,\varepsilon)\rightarrow f^0_{2}(\phi,\psi)
=\left(\frac{\lambda^{2}_{j}}{a^{2}_{2}}h_{20}+\frac{\lambda_{j}}{{a_{2}}}h^*_{2}\right)
(a_{2}-b_{2}h_{10})-\frac{\lambda_{j}}{a_{2}}b_{2}h_{20}h^*_1\\[2mm]
=
\left(\frac{\lambda^{2}_{j}}{a^{2}_{2}}(\psi-\beta\phi)
+\frac{\lambda_{j}\beta\phi}{a_{2}(\psi-\beta\phi)}\right)
\left(a_{2}-b_{2}\frac{\phi}{\psi-\beta\phi}\right)
+\frac{\lambda_{j}b_{2}}{a_{2}}\frac{\phi}{\psi-\beta\phi}
\left(1+\frac{\beta\phi}{\psi-\beta\phi}\right)\\[2mm]
=
\frac{\lambda^{2}_{j}}{a_{2}}(\psi-\beta\phi)
-\frac{\lambda^{2}_{j}b_{2}\phi}{a^{2}_{2}}
+\left(\beta+\frac{b_{2}}{a_{2}}\right)\lambda_{j}\frac{\phi}{\psi-\beta\phi}.
\end{array}
\end{equation}
For the fixed $j\ge 1$, consider the system of positive steady states of (\ref{1})
\begin{equation}\label{2}
\begin{cases}
\Delta\phi+\varepsilon f_{1}(\phi,\psi,\varepsilon)=0 \ \ &\mbox{in}\;
\Omega,\\
\Delta\psi+\lambda_{j}\psi
-\left(\beta +\frac{b_2}{a_2}\right)\lambda_j\phi +\varepsilon f_{2}(\phi,\psi,\varepsilon)=0\ \ &\mbox{in}\;
\Omega,\\
\frac{\partial\phi}{\partial\nu}=\frac{\partial\psi}{\partial\nu}=0\ \ &\mbox{on}\;
\partial\Omega ;
\end{cases}
\end{equation}
or
\begin{equation}\label{L3}
    \mathcal{L}\biggl[\begin{array}{c}
    \phi\\
    \psi\end{array}\biggr]+\varepsilon\biggl[\begin{array}{l}
    f_{1}(\phi,\psi,\varepsilon )\\
    f_{2}(\phi,\psi, \varepsilon )
    \end{array}\biggr]=
    \biggl[\begin{array}{c}
    0\\
    0\end{array}\biggr],
\end{equation}
with
\begin{equation*}
    \mathcal{L}=
\biggl[\begin{array}{cc}
\Delta & 0\\
-\left(\beta+\frac{b_{2}}{a_{2}}\right)\lambda_{j} & \Delta +\lambda_{j}
\end{array}\biggr].
\end{equation*}
The kernel of $\mathcal{L}$
consists of  the following  two-dimensional space:
\begin{equation}\label{2.9}
\mbox{Ker}(\mathcal{L})=\mbox{Span}\,\biggl\{\,
\boldsymbol{e}_{1}:=
\biggl(1, \beta+\frac{b_{2}}{a_{2}}\biggr),\quad
\boldsymbol{e}_{2}:=(0,1)\varphi_{j}\,\biggr\};
\end{equation}
where $\varphi_{j}(x)$ is an eigenfunction satisfying
\begin{equation}\label{egf}
-\Delta\varphi_{j}=\lambda_{j}\varphi_{j}\ \ \mbox{in}\ \Omega,
\quad \frac{\partial\varphi_{j}}{\partial \nu}=0\ \ \mbox{on}\ \partial\Omega,
\quad \|\varphi_j\|_{L^2(\Omega)}=1.
\end{equation}
Under the assumption that $N\le 4$ and $a_{1}/a_{2}> b_{1}/b_{2}$ and
$\lambda_{j}>0$ is a simple eigenvalue, by applying
the Lyapunov-Schmidt reduction method, our proof will find
 the existence of
$(s_{j,0}^{+}, \mu_{j,0}^{+})$
and
$(s_{j,0}^{-}, \mu_{j,0}^{-})$
such that $s^{\pm}_{j,0}>0$, $\mu^{-}_{j,0}<0<\mu^{+}_{j,0}$ and the solution $(\phi(x), \psi(x))$ of
the limiting system of \eqref{2} at $\varepsilon\to 0$ satisfying
$$
(\phi(x), \psi(x))-
s^{\pm}_{j,0}
\biggl(1, \beta+\frac{b_{2}}{a_2}(1+\mu^{\pm}_{j,0}\varphi_{j}(x))\biggr)\in {\rm Range}(\mathcal{L}).
$$

Next, for each $j$ and sign of $\pm$,
we use the Implicit Function Theorem to construct
a curve $(\phi_{\varepsilon }, \psi_{\varepsilon })$
to the  system \eqref{2}
with small enough $\varepsilon >0$.
It should be noted that
the profile of
$$
h_{j}(\mu )=\dfrac{\int_{\Omega}(1+\mu\varphi_{j}(x))^{-2}{\rm d}x}
{\int_{\Omega}(1+\mu\varphi_{j}(x))^{-1}{\rm d}x}$$
obtained by \cite{LNY2} plays an important role
in solving the limiting system of \eqref{2} in higher dimensional case ($N>1$) at $\varepsilon=0 $
and
checking a non-degenerate condition
for use of the Implicit Function Theorem.
By the transformation \eqref{change},
this curve realizes the bifurcation branch of
\eqref{LNlim2} on which the
$w$ component blows up
with the rate $O(1/\varepsilon )$ as
$\varepsilon:=a_{2}\lambda_{j}^{-1}-d_{2}>0$ tends to zero.

The contents of this paper are as follows:
In Section 2, we state the  main existence and stability results  of this paper.
In Section 3, we mainly discuss the construction of nonconstant
solutions of \eqref{2}.
In Section 4, we show the instability of the solution branches obtained in
the previous section. In Section 5, we investigate  the existence and stability of the
nonconstant steady states of the original SKT model \eqref{SKT} with large enough $\alpha$.

The following Theorems \ref{LSprop} and \ref{blthm} are concerned with the existence and asymptotic
behavior of the solution branches $(\phi^{\pm}_{j,\varepsilon}(x), \psi^{\pm}_{j,\varepsilon}(x))$ of \eqref{2} and $\varGamma_{j,\pm}$
of the shadow system \eqref{LNlim2} near the blow-up point
$d_{2}=a_{2}/\lambda_{j}$.
Theorem \ref{blthm} shows that each branch of $\varGamma_{j,\pm}$
can be parameterized by $\varepsilon :=a_{2}\lambda_{j}^{-1}-d_{2}>0$,
and then, the $w$ component of each branch blows up
with the rate $O(\varepsilon^{-1})$ as $\varepsilon\to 0$.
\begin{thm}\label{LSprop}
Assume that $N\le 4$, $a_{1}/a_{2}>b_{1}/b_{2}$ and $\lambda_{j}$ is
simple. Denote $\varepsilon =a_{2}\lambda_{j}^{-1}-d_{2}$.
Then for any given $d_1>0$, $\beta\ge 0$ and each sign of $\pm$,
there exists $\overline{\varepsilon}>0$
such that, for any $\varepsilon\in [0, \overline{\varepsilon}]$, the system \eqref{2}
has a positive solution
$(\phi^{\pm}_{j,\varepsilon}(x), \psi^{\pm}_{j,\varepsilon}(x))$
satisfying
\begin{equation}\label{ppe}
    \biggl[
    \begin{array}{c}
    \phi^\pm_{j,\varepsilon}(x)\\
    \psi^\pm_{j,\varepsilon}(x)
    \end{array}
    \biggr]
    =s^{\pm}_{j}(\varepsilon )\biggl(
    \biggl[
    \begin{array}{c}
    1\\
    \beta +\frac{b_2}{a_2}
    \end{array}
    \biggr]
    +\mu^{\pm}_{j}(\varepsilon )
    \biggl[
    \begin{array}{c}
    0\\
    \frac{b_2}{a_2}\varphi_{j}(x)
    \end{array}
    \biggr]\biggr)
    +\varepsilon\biggl[
    \begin{array}{c}
    \widetilde{\phi}^{\pm}_{j}(x,\varepsilon )\\
    \widetilde{\psi}^{\pm}_{j}(x,\varepsilon )
    \end{array}
    \biggr],
\end{equation}
where $\boldsymbol{\varPhi}^{*\pm}_{j}(x,\varepsilon ):=
(\widetilde{\phi}^{\pm}_{j}(x,\varepsilon ), \widetilde{\psi}^{\pm}_{j}(x,\varepsilon ))\in {\rm Range}(\mathcal{L})$
and $(s^{\pm}_{j}(\varepsilon ), \mu^{\pm}_{j}(\varepsilon ))\in \mathbb{R}^{2}$ are functions of $C^{1}$-class
for $\varepsilon\in [0,\overline{\varepsilon}]$
satisfying
\begin{equation}
        (s^{\pm}_{j}(0),\mu^{\pm}_{j}(0))= (s_{j,0}^{\pm }, \mu^{\pm}_{j,0})
        \ \ \mbox{and}\ \
        \boldsymbol{\varPhi}^{*\pm}_{j}(x,0)= -\mathcal{L}_{X_0}^{-1}(I-P)F_{0}(s_{0},t_{0}),
\nonumber
\end{equation}
where $\mu^{\pm}_{j,0}$ is obtained in Proposition \ref{prop1}, $s^{\pm}_{j,0}$ is defined by \eqref{sexp},
and $\boldsymbol{\varPhi}^{*\pm}_{j}(x,0)$ is defined in Section 3.
\end{thm}

\begin{thm}\label{blthm}
Suppose that $N\le 4$, $a_{1}/a_{2}>b_{1}/b_{2}$ and
$\lambda_{j}>0$ is simple.
For any such $j$ and each sign of $\pm$,
there exists a small positive number $\overline{\varepsilon}_{j,\pm}>0$
such that the 2nd shadow system \eqref{LNlim2} with $\beta \ge 0$, $d_1>0$ and $d_2=a_{2}\lambda_{j}^{-1}-\varepsilon>0$
admits a curve
\begin{equation}\label{vG}
\varGamma_{j,\pm}:=\biggl\{
(u,w,d_{2})=\biggl(u^\pm_{j,\varepsilon}(x), w^\pm_{j,\varepsilon}(x),
\dfrac{a_{2}}{\lambda_{j}}-\varepsilon \biggr)\,:\,
\varepsilon\in (0,\overline{\varepsilon}_{j,\pm})\biggr\}
\end{equation}
of nonconstant positive solutions.
Furthermore, $\varGamma_{j,\pm}$ forms a unbounded curve
parameterized by $\varepsilon\in (0,\overline{\varepsilon}_{j,\pm}]$, which is continuously differentiable in $\varepsilon$ and
has the following asymptotic behavior
\begin{equation}\label{asmint}
\begin{cases}
u^\pm_{j,\varepsilon}(x)\to u^\pm_{j,0}(x)=\frac{a_2}{b_2(1+\mu^\pm_{j,0}\varphi_{j}(x))}>0,
& {\rm as}\;\varepsilon \downarrow 0,\\
\varepsilon
 w^\pm_{j, \varepsilon}(x)\to   \frac{b_2}{a_2}s^\pm_{j,0}(1+\mu^\pm_{j,0}\varphi_{j}(x))>0,
&{\rm as}\;\varepsilon \downarrow 0;
\end{cases}
\end{equation}
where $\mu^\pm_{j,0}$ is defined in  Proposition \ref{prop1} and $s^\pm_{j,0}>0$ is defined by \eqref{sexp}.
\end{thm}
From the view-point of the bifurcation theory,
it can be said that Theorem \ref{blthm} shows a
phenomenon so called
bifurcation from infinity at $d_{2}=a_{2}/\lambda_{j}$.

In Section 4, we investigate the stability/instability  of each steady state $(u^\pm_{j,\varepsilon}(x)$, $w^\pm_{j,\varepsilon}(x))$
for the time-evolutional  2nd shadow system \eqref{2.2*} with small $\varepsilon=a_{2}\lambda_{j}^{-1}-d_{2}>0$.
It suffices to investigate the eigenvalue problem corresponding to the linearized  system of \eqref{1} around the steady state $(\phi^{\pm}_{j,\varepsilon}(x), \psi^{\pm}_{j,\varepsilon}(x))$,  which  can be expressed  in the following form:
\begin{equation}\label{eigen}
 \sigma T^{\varepsilon}(x)
\biggl[
\begin{array}{c}
\widehat{\phi}(x)\\
\widehat{\psi}(x)
\end{array}
\biggr]
=
\mathcal{L}
\biggl[
\begin{array}{c}
\widehat{\phi}(x)\\
\widehat{\psi}(x)
\end{array}
\biggr]
+\varepsilon
F^{\varepsilon}_{(\phi_{\varepsilon}, \psi_{\varepsilon})}(x)
\biggl[
\begin{array}{c}
\widehat{\phi}(x)\\
\widehat{\psi}(x)
\end{array}
\biggr].
\end{equation}
By applying  perturbation argument, it is naturally  expected  that for each small $\varepsilon>0$  the eigenvalue problem \eqref{eigen} has at least one eigenvalue near zero, the sign of the real part of the eigenvalue near zero may determine the stability/instability of the steady state. In Section 4, by seeking the precise location of an  eigenvalue near zero and applying special Lyapunov-Schmidt decomposition estimates on the corresponding eigenfunction, we prove the spectral instability and nonlinear instability of all the branches of steady states of \eqref{2.2*} with  small enough $\varepsilon>0$.
\begin{thm}\label{instthm}
Let $(u^\pm_{j,\varepsilon}(x), w^\pm_{j,\varepsilon}(x))$
be the steady states of \eqref{2.2*} with
$d_{2}=a_{2}\lambda_{j}^{-1}-\varepsilon $  obtained in Theorem \ref{blthm}.
Then for small enough $\varepsilon >0$,
$(u^\pm_{j,\varepsilon}(x), w^\pm_{j,\varepsilon}(x))$
are spectrally unstable; and the  eigenvalue problem \eqref{eigen} has a positive  eigenvalue $\sigma_{\varepsilon}=\varepsilon \Lambda_{\varepsilon}>0$ with $\Lambda_{\varepsilon}\to \lambda_j$ as $\varepsilon\to 0$.
\end{thm}

By the perturbation of solutions of the 2nd shadow
system \eqref{LNlim2}, in Section 5
we construct the curves of nonconstant positive solutions
to the original  stationary SKT model \eqref{sSKT}
with sufficiently large $\alpha$ and small $a_{2}\lambda_{j}^{-1}-d_{2}>0$, and show that the obtained positive steady states of the SKT model \eqref{SKT}  are nonlinearly and spectrally unstable.

\begin{thm}\label{SKTthm}
Assume that $N\le 4$,
$a_{1}/a_{2}>b_{1}/b_{2}$ and $\lambda_{j}>0$ is simple. For any such $j$ and each sign of $\pm$,
there exists $\overline{\varepsilon}>0$ such that
for any $\underline{\varepsilon}\in (0,\overline{\varepsilon })$,
there exists a large
$\overline{\alpha}^\pm_{j,\underline{\varepsilon}}>0$ such that
if $\varepsilon\in (\underline{\varepsilon },
\overline{\varepsilon })$,
the  SKT model \eqref{SKT} with fixed
$\beta \ge 0$, $d_1>0$ and $\alpha>\overline{\alpha}^\pm_{j,\underline{\varepsilon}}$
admits a branch of nonconstant positive steady states
$$
\varGamma_{j,\pm}^{(\alpha )}=\Bigl\{\,
(u,v,\alpha):=
\Bigl(u^\pm_{j,\varepsilon,\alpha}(x),
v^\pm_{j,\varepsilon,\alpha}(x),
\,\alpha\Bigr)\,:\,
\alpha >\overline{\alpha}^\pm_{j,\underline{\varepsilon}}\,
\Bigr\}.
$$
Here $(u^\pm_{j, \varepsilon,\alpha}(x ),
v^\pm_{j,\varepsilon,\alpha}(x))$ satisfy
\begin{equation}\label{uni}
\lim\limits_{\alpha\to\infty}
(u^\pm_{j, \varepsilon,\alpha}(x),
\alpha v^\pm_{j,\varepsilon,\alpha}(x))=
(u^\pm_{j, \varepsilon}(x),
w^\pm_{j,\varepsilon}(x))
\ \ \mbox{uniformly in}\ \overline{\Omega},
\end{equation}
with $(u^\pm_{j, \varepsilon}(x), w^\pm_{j,\varepsilon}(x),\varepsilon )\in\varGamma_{j,\pm}$
defined by \eqref{vG}.
Furthermore, $(u^\pm_{j, \varepsilon,\alpha}(x),
v^\pm_{j, \varepsilon,\alpha}(x))$
 are spectrally unstable.
\end{thm}
\begin{remark}
For the special case when $\Omega$ is one or two dimensional, it is well known that for any fixed nonnegative initial data the SKT evolutional system \eqref{SKT} with $\beta=0$ and $\alpha>0$ admits a uniformly bounded global solution in time, however the spectral and nonlinear instability of the lower branches of positive steady states obtained in  Theorem \ref{SKTthm} when $d_2$ is slightly less than  $a_2/\lambda_j$   indicate that no matter how closeness of the initial data to the unstable steady state, there always exist global solutions  staying away from the unstable steady state for large $t$,  the longtime behavior of such  solutions in time are still open problems. Also note that for the case when $d_2$ is slightly less than $a_2/\lambda_1$  in \cite{LNY2} it was  proved that the SKT model \eqref{SKT} with $\beta=0$ and large $\alpha$ has another two (upper) branches of  steady states perturbed from the 1st shadow system \eqref{LNlim2} and they are stable steady states, which attract all the solutions with small initial perturbation of such steady states. For the other cases when $\Omega$ is higher dimensional or $\beta>0$, it is unknown that whether the SKT model \eqref{SKT} still admit  global solutions in time when the initial data are near the unstable lower branches of steady states, the solutions may blow-up in finite time or in infinite time or tend to other types of steady states.
\end{remark}
\begin{remark}
From the expression of  $s^{\pm}_{j}(0)$, $\mu^{\pm}_{j}(0)$ and $\boldsymbol{\varPhi}^{*\pm}_{j}(x,0)$ obtained in Section 3,  we can see that $\mu^{\pm}_{j}(0)$ is independent of $d_1$,  $s^{\pm}_{j}(0)$ has a positive  lower bound for all $d_1>0$, and $s^{\pm}_{j}(0)\sim \frac{C_1}{d_1}$ as $d_1\to 0$, which with
 Theorems \ref{blthm} and \ref{SKTthm} indicate that for any  $d_1>0$ and $\beta\ge 0$, the 2nd shadow system \eqref{LNlim2} with small $a_{2}\lambda_{j}^{-1}-d_{2}>0$ admit several branches of positive solutions $(u^\pm_{j, \varepsilon}(x),
w^\pm_{j,\varepsilon}(x))$ with  bounded  $u$-component (uniformly in both $d_1$ and $d_2$)  and blowing up $w$-component (as
 $d_2 \uparrow a_{2}/\lambda_{j}$ or $d_1\downarrow 0$). Also note that $s^{\pm}_{j,0}\to s^{\pm *}_{j,0}>0$, $\boldsymbol{\varPhi}^{*\pm}_{j}(x,0)\to (0, \widetilde{\psi}^{\pm*}_{j}(x,\varepsilon ))$ as $d_1\to +\infty$;
thus the expressions of \eqref{ppe} in Theorem \ref{LSprop} and \eqref{asmint} in Theorem \ref{blthm}  have finite limits for the limiting case when
$d_1\to +\infty$, and the limiting asymptotic  expression  of the blowing up solution branches in \eqref{ppe} and \eqref{asmint} in the limiting case $d_1\to +\infty$ coincide with the asymptotic expression of the blowing-up solution branches   obtained in \cite{LW2} and \cite{Tang} for the 3rd shadow system \eqref{LNlim3} (the limiting system of \eqref{LNlim2} when $d_1\to +\infty$ after some transformation).

  For the limiting case when $d_1\to +\infty$, it is known that the  positive  solution $(u(x),w(x))$ of the 2nd shadow system \eqref{LNlim2} satisfies $u(x)(1+w(x))\to \tau$ (a positive constant), where $(\tau ,w(x))$ satisfies the  3rd shadow system \eqref{LNlim3}. In \cite{LW2} and \cite{Tang}, by applying special  transformation to the 3rd shadow system \eqref{LNlim3} and applying Lyapunov-Schmidt reduction method (where the kernel of the limiting  linearized operator can be one-dimensional after some transformation), the authors proved the existence and asymptotic behavior  of some blowing up branches of positive solutions $(\tau, w(x))$ to the 3rd shadow system \eqref{LNlim3}. It is worth mentioning that the techniques and perturbation  argument applied in \cite{LW2} and \cite{Tang} to the 3rd shadow system or to the corresponding  perturbed  stationary SKT competition model \eqref{sSKT}   (when both $d_1$ and $\alpha$ are large enough) are not applicable to the case when $d_1$ is not large enough. As for the 2nd shadow  system \eqref{LNlim2} with bounded $d_1>0$,
the higher degeneracy of the limiting operator $\mathcal{L}$ (with two dimensional $\mbox{Ker}\,(\mathcal{L})$)
 induces some technical difficulties and complexity in applying the Lyapunov-Schmidt reduction method for obtaining the detailed  asymptotic structure and doing stability analysis of the blowing up steady states.
\end{remark}

\section{Existence and parameterization of the solution branches to the 2nd shadow system \eqref{LNlim2} near the blow-up points}

As formulated in Section 2, in order to magnify a neighborhood of each blow-up point at
$d_{2}=a_{2}/\lambda_{j}$ by rescaling,
we employ the rescaling transformation \eqref{change}
in the 2nd shadow system \eqref{LNlim2}, the
new unknown functions
$\phi$ and $\psi$
satisfy the Neumann boundary value problem \eqref{2} or \eqref{L3}, which can be expressed as the following nonlinear system with small parameter $\varepsilon>0$:
\begin{equation}\label{F=0}
    \mathcal{L}\biggl[\begin{array}{c}
    \phi\\
    \psi\end{array}\biggr]+\varepsilon\biggl[\begin{array}{l}
    f_{1}(\phi,\psi,\varepsilon )\\
    f_{2}(\phi,\psi, \varepsilon )
    \end{array}\biggr]=
    \biggl[\begin{array}{c}
    0\\
    0\end{array}\biggr],
\end{equation}
where the linear operator $\mathcal{L}: X\to Y$ is defined by
\begin{equation}\label{Ldef}
\mathcal{L}:=\biggl[
\begin{array}{cc}
\Delta & 0\\
-\left(\beta+\frac{b_{2}}{a_2}\right)\lambda_{j} & \Delta +\lambda_{j}
\end{array}
\biggr],
\end{equation}
with $X:= W^{2,p}_{\nu}(\Omega )\times W^{2,p}_{\nu }(\Omega )$ and $Y:=L^{p}(\Omega)\times L^{p}(\Omega)$ for $p>1$;
$f_1$ and $f_2$ are defined by
\begin{equation}\label{fgdef}
\begin{split}
&f_{1}(\phi, \psi, \varepsilon ):=\frac{h_{1\varepsilon}}{d_1}
(a_{1}-b_1h_{1\varepsilon}),\\
&f_{2}(\phi, \psi, \varepsilon ):=\frac{1}{\varepsilon}\left[\frac{h_{2\varepsilon}}{d_2}
(a_{2}-b_{2}h_{1\varepsilon})-\frac{h_{20}\lambda_j}{a_2}(a_{2}-b_{2}h_{10})\right],
\end{split}
\end{equation}
where $(h_{1\varepsilon }, h_{2\varepsilon })$ and
$(h_{10}, h_{20})$ are defined by \eqref{2.5} and \eqref{h},
respectively.

%%%%%%%%%%%%%%%%%%%%%%%%%%%%%%%%%%%%%%%%%%%%%%%%%%%%%%%%%%%%%%%%%%
\subsection{Lyapunov-Schmidt reduction scheme}
In what follows, we use the Lyapunov-Schmidt reduction method
to construct the set of solutions to the system
\eqref{F=0} for small $\varepsilon>0$.
In a crucial step of our proof of Theorem \ref{LSprop},
we shall derive the limiting set of solutions of \eqref{F=0}
as $\varepsilon\to 0$.
More precisely, we shall
construct the set of positive solutions of
\eqref{F=0}
as $\varepsilon\to 0$ (see Theorem \ref{LSprop}).

In order to obtain the kernel of the
linear part $\mathcal{L}\,:\,X\to Y$
of \eqref{Ldef},
we solve the Neumann boundary value problem of the following
linear elliptic system:
\begin{equation}\label{ker}
\begin{cases}
\Delta\phi =0\ \ &\mbox{in}\ \Omega,\\
\Delta\psi +\lambda_{j}\psi -\left(\beta+\frac{b_{2}}{a_2}\right)\lambda_{j}\phi=0
\ \ &\mbox{in}\ \Omega,\\
\frac{\partial \phi}{\partial \nu}=\frac{\partial \psi}{\partial \nu}=0
\ \ &\mbox{on}\
\partial\Omega.
\end{cases}
\end{equation}
Since
$\phi=\tau$
(constant)
from the first equation of \eqref{ker},
then we substitute it into the second equation of \eqref{ker} to see
$$-\Delta\psi(x) =\lambda_{j}\psi(x)-\lambda_j\left(\beta+\dfrac{b_{2}}{a_{2}}\right)\tau
\ \ \mbox{in}\ \Omega,\qquad
\frac{\partial \psi}{\partial \nu}(x)=0\ \ \mbox{on}\ \partial\Omega.
$$
Hence it follows that
$$
\psi(x)-\left(\beta+\dfrac{b_{2}}{a_{2}}\right)\tau=s\varphi_{j}(x),
\ \ \mbox{that is,}\ \
\psi(x)=\left(\beta+\dfrac{b_{2}}{a_{2}}\right)\tau+s\varphi_{j}(x),
$$
where
$s\in\mathbb{R}$ and
$\varphi_{j}(x)$ is
an eigenfunction satisfying \eqref{egf}.
Thus $\mbox{Ker}(\mathcal{L})$
consists of  the following two-dimensional space:
\begin{equation}\label{kerH}
\mbox{Ker}(\mathcal{L})=\mbox{Span}\,\left\{\,
\boldsymbol{e}_{1}:=
\left(1, \beta+\frac{b_{2}}{a_2}\right),\ \
\boldsymbol{e}_{2}:=
(0,1)
\varphi_{j}(x)\,\right\}.
\end{equation}
Here it follows that
$
\langle\boldsymbol{e}_{1},\boldsymbol{e}_{2}\rangle:=
\int_{\Omega}\boldsymbol{e}_{1}\cdot
\boldsymbol{e}_{2}=0$.
By the Fredholm alternative theorem,
we see that
$\mbox{Range}(\mathcal{L})=(\,\mbox{Ker}(\mathcal{L}^{*})\,)^\bot$ in $Y$,
where $\mathcal{L}^{*}\,:\,X\to Y$ is the adjoint operator of $\mathcal{L}$
expressed as
$$
\mathcal{L}^{*}=
\biggl[\begin{array}{cc}
\Delta & -(\beta+\frac{b_2}{a_2})\lambda_{j}\\
0 & \Delta +\lambda_{j}
\end{array}\biggr].
$$
It is easily verified that
\begin{equation}\label{kerHs}
\mbox{Ker}\,(\mathcal{L}^{*})
=
\mbox{Span}\,\biggl\{
\boldsymbol{e}_{1}^{*}:=\biggl(\dfrac{1}{|\Omega|},\,0\biggr),\quad
\boldsymbol{e}_{2}^{*}:=\biggl(-\beta-\dfrac{b_{2}}{a_{2}},\,1\biggr)\varphi_{j}(x)
\biggr\}.
\end{equation}
and thereby,
$$
\langle\boldsymbol{e}_{1},
\boldsymbol{e}_{1}^{*}\rangle=1,\qquad
\langle\boldsymbol{e}_{2},
\boldsymbol{e}_{2}^{*}\rangle=1,\qquad
\langle\boldsymbol{e}_{1},
\boldsymbol{e}_{2}^{*}\rangle=\langle\boldsymbol{e}_{2},
\boldsymbol{e}_{1}^{*}\rangle=0.
$$
For applying the Lyapunov-Schmidt reduction method,
we employ
the following direct sum decompositions
of $X$ and $Y$:
\begin{equation}\label{decom}
X=\mbox{Ker}(\mathcal{L})\oplus X_0
\quad\mbox{and}\quad
Y=\mbox{Ker}(\mathcal{L})\oplus Y_0,
\end{equation}
where
$X_0=\mbox{Range}(\mathcal{L})\cap X$, $Y_0=\mbox{Range}(\mathcal{L})$ and
$\mbox{Ker}(\mathcal{L})=\mbox{Span}\{\boldsymbol{e}_{1}, \boldsymbol{e}_{2}\}$.

By virtue of \eqref{decom},
we introduce the projection
$P\,:\,Y\to \mbox{Ker}(\mathcal{L}) $
%(or $X\to\mbox{Ker}(\mathcal{L}))$
as
\begin{equation}\label{Pdef}
P\,
\boldsymbol{\varPhi}
=
s\boldsymbol{e}_{1}+t\boldsymbol{e}_{2}
\quad\mbox{for}\quad
\boldsymbol{\varPhi}=(\phi, \psi)
\in Y,
%\,(\,\mbox{or}\ X)
\end{equation}
where
\begin{equation}\label{st}
\begin{array}{l}
s=\dfrac{\langle\boldsymbol{\varPhi},\boldsymbol{e}_{1}^{*}\rangle}
{\langle\boldsymbol{e}_{1},\boldsymbol{e}^{*}_{1}\rangle}
=\dfrac{1}{|\Omega |}\displaystyle\int_{\Omega }\phi(x)\,\dx,\\[4mm]
t=\dfrac{\langle\boldsymbol{\varPhi},\boldsymbol{e}_{2}^{*}\rangle}
{\langle\boldsymbol{e}_{2},\boldsymbol{e}^{*}_{2}\rangle}
=\displaystyle\int_{\Omega }\left[-\left(\beta+\frac{b_2}{a_2}\right)\phi(x)+\psi(x)\right]\,\varphi_{j}(x)\,\dx.
\end{array}
\end{equation}
Furthermore,
the projections
$P_{i}\,:\,Y\to\mbox{Span}\{\boldsymbol{e}_{i}\}$,\ $i=1,2$
are defined by
\begin{equation}\label{P1P2}
P_{1}\,\boldsymbol{\varPhi}=
s\boldsymbol{e}_{1},
\qquad
P_{2}\,\boldsymbol{\varPhi}=
t\boldsymbol{e}_{2}.
\end{equation}
Hence it follows that
$P=P_{1}+P_{2}$
and the operator
$I-P$
gives a projection from $Y$ to $Y_0=\mbox{Range}(\mathcal{L})$.
Then we decompose the unknown function
$\boldsymbol{\varPhi}=(\phi,\psi)\in X$
of the system \eqref{F=0} into
$\mbox{Span}\{\boldsymbol{e}_{1}\}$,
$\mbox{Span}\{\boldsymbol{e}_{2}\}$
and
$X_0$
components as
\begin{equation}\label{decomp}
    \boldsymbol{\varPhi}(x)=s\boldsymbol{e}_{1}(x)+t\boldsymbol{e}_{2}(x)+\varepsilon\boldsymbol{\varPhi}^*(x),
\end{equation}
where $(s,t)$ is defined by \eqref{st} and
$\varepsilon\,\boldsymbol{\varPhi}^*=(I-P)\boldsymbol{\varPhi}\in X_0$.

We set $t=\frac{b_{2}}{a_{2}}s\mu $ for $s\ne 0$, then
\begin{equation}\label{2.15}
    s\boldsymbol{e}_{1}+t\boldsymbol{e}_{2}=
    s
    \biggl[\begin{array}{c}
    1\\
    \beta+\frac{b_{2}}{a_{2}}(1+\mu\varphi_{j}(x))
    \end{array}\biggr],
\end{equation}
thus
\begin{equation}\label{2.16}
    \biggl[\begin{array}{c}
    \phi(x)\\
    \psi(x)
    \end{array}\biggr]=s
    \biggl[\begin{array}{c}
    1\\
    \beta+\frac{b_{2}}{a_{2}}(1+\mu\varphi_{j}(x))
    \end{array}\biggr]+\varepsilon
    \biggl[\begin{array}{c}
    \phi^*(x)\\
    \psi^*(x)
    \end{array}\biggr]
\end{equation}
with
$(\phi^*(x),\psi^*(x)):=\boldsymbol{\varPhi}^*(x)\in X_0=\mbox{Range}(\mathcal{L})\cap X.$

Denote
$\boldsymbol{\varPhi}^0=(\phi^{0},\psi^{0})
= s(1,\;\beta+\frac{b_{2}}{a_{2}}(1+\mu\varphi_{j}))$.
Using the Taylor expansion of $f_{i}(\phi,\psi,\varepsilon)$ and the
expression~\eqref{2.16}, it is easy to check that
\begin{equation}\label{2.17}
\begin{split}
f_{1}(\phi^{0}+\varepsilon\phi^*,\;\psi^{0}+\varepsilon\psi^*,\;\varepsilon)
&= f^{0}_{1}(\phi^{0},\psi^{0})
   +\varepsilon f^{1}_{1}(s, \mu,\phi^*,\psi^*,\varepsilon)\\
&= \frac{1}{d_{1}}\,
   \frac{a_{2}}{b_{2}(1+\mu\varphi_{j})}
   \left(a_{1}
   -b_{1}\frac{a_{2}}{b_{2}(1+\mu\varphi_{j})}\right)\\
&\quad+\varepsilon f^{*}_{1}(s,\mu,\phi^*,\psi^*,\varepsilon),
\end{split}
\end{equation}
and
\begin{equation}\label{2.18}
\begin{split}
f_{2}(\phi^{0}+\varepsilon\phi^*,\;\psi^{0}+\varepsilon\psi^*,\;\varepsilon)
&= f^{0}_{2}(\phi^{0},\psi^{0})
   +\varepsilon f^{1}_{2}(s, \mu,\phi^*,\psi^*,\varepsilon)\\
&= \frac{\lambda_{j}^{2}}{a_{2}}(\psi^{0}-\beta\phi^{0})
   -\frac{\lambda_{j}^{2}b_{2}}{a_{2}^{2}}\phi^{0}
   +\lambda_{j}\Bigl(\beta+\dfrac{b_{2}}{a_{2}}\Bigr)
     \frac{\phi^{0}}{\psi^{0}-\beta\phi^{0}}\\
&\quad+\varepsilon f^{*}_{2}(s, \mu,\phi^*,\psi^*,\varepsilon)\\[1ex]
&= s\,\frac{\lambda_{j}^{2}b_{2}}{a_{2}^{2}}\,\mu\varphi_{j}
   + \frac{\lambda_{j}(a_{2}\beta+b_{2})}{b_{2}(1+\mu\varphi_{j})}\\
&\quad+\varepsilon f^{*}_{2}(s, \mu,\phi^*,\psi^*,\varepsilon),
\end{split}
\end{equation}
with some bounded functions
$f^{*}_{j}(s, \mu, \phi^{*}, \psi^{*}, \varepsilon)$
of $C^1$ class,
where $f^0_{1}$ and $f^0_2$ are defined by \eqref{4}.

For small $\varepsilon>0$, plugging \eqref{2.16}-\eqref{2.18} into the system  \eqref{F=0},
then the system   \eqref{F=0} becomes the following system of $(s_,\mu,\phi^*,\psi^*)$
\begin{equation}\label{2.19}
        -\mathcal{L}\biggl[
    \begin{array}{l}
    \phi^*(x)\\
    \psi^*(x)
    \end{array}\biggr]
    =(I-P)\biggl[\begin{array}{c}
    f_{1}(\phi^{0}+\varepsilon\phi^*, \psi^{0}+\varepsilon\psi^*,\varepsilon)\\
    f_{2}(\phi^{0}+\varepsilon\phi^*, \psi^{0}+\varepsilon\psi^*,\varepsilon)
    \end{array}
    \biggr],
\end{equation}
and
\begin{equation}\label{2.20}
    P\biggl[\begin{array}{c}
    f_{1}(\phi^{0}+\varepsilon\phi^*,\psi^{0}+\varepsilon\psi^*,\varepsilon)\\
    f_{2}(\phi^{0}+\varepsilon\phi^*,\psi^{0}+\varepsilon\psi^*,\varepsilon)
    \end{array}
    \biggr]=0,
\end{equation}
with
$\boldsymbol{\varPhi}^0=(\phi^{0},\psi^{0})
= s(1,\;\beta+\frac{b_{2}}{a_{2}}(1+\mu\varphi_{j}))$
for $s,\mu\in \mathbb{R}$
and
$(\phi^*(x),\psi^*(x))\in X_0$.

We denote the linear operator $\mathcal{L}$ with the domain
restricted to $X_0$ by
$\mathcal{L}_{X_0} \; :\,X_0\to Y_0=\mbox{Range}(\mathcal{L}_{X_0})$, then
 $\mathcal{L}_{X_0}$ is an isomorphism from $X_0$ to $Y_0$.
By (\ref{2.17}) and (\ref{2.18}), we can rewrite the equations
(\ref{2.19}) and (\ref{2.20}) as
\begin{equation}\label{2.21}
    \begin{cases}
    P_{1}(F_{0}(s,\mu)+\varepsilon
    F_{1}(s,\mu,\boldsymbol{\varPhi}^*,\varepsilon))=0,\\
    P_{2}(F_{0}(s,\mu)+\varepsilon
    F_{1}(s,\mu,\boldsymbol{\varPhi}^*,\varepsilon))=0,\\
    \boldsymbol{\varPhi}^*=-\mathcal{L}_{X_0}^{-1}(I-P)(F_{0}(s,\mu)+\varepsilon
    F_{1}
    (s,\mu,\boldsymbol{\varPhi}^*,\varepsilon)),
    \end{cases}
\end{equation}
where
\[
F_{0}(s,\mu):=\left[\begin{array}{c}f^{0}_{1}(\mu),\\
f_{2}^{0}(s, \mu)\end{array}\right],\quad
F_{1}(s, \mu, \boldsymbol{\varPhi}^{*}, \varepsilon )=
\left[\begin{array}{c}
f^{*}_{1}(s,\mu, \boldsymbol{\varPhi}^{*}, \varepsilon )\\
f^{*}_{2}(s,\mu, \boldsymbol{\varPhi}^{*}, \varepsilon )\\
\end{array}
\right]
\]
with
\begin{equation}\label{2.22}
\begin{cases}
f^0_1(\mu):=
    \dfrac{1}{d_{1}}\dfrac{a_{2}}{b_{2}(1+\mu\varphi_{j}(x))}\left(a_{1}
    -b_{1}\dfrac{a_{2}}{b_{2}(1+\mu\varphi_{j}(x))}\right), \vspace{2mm} \\
f^0_2(s, \mu):=
    s\lambda^{2}_{j}\dfrac{b_{2}}{a^{2}_{2}}\mu\varphi_{j}(x)+
    \dfrac{\lambda_{j}(a_{2}\beta+b_{2})}{b_{2}(1+\mu\varphi_{j}(x))}.
\end{cases}
\end{equation}

As a framework to solve \eqref{2.21},
we define
$g_{j}(s, \mu, \boldsymbol{\varPhi}^{*}, \varepsilon )\in\mathbb{R}$
$(j=1,2)$ and
$g_{3}(s, \mu, \boldsymbol{\varPhi}^{*}, \varepsilon )\in X_0$
by
\begin{equation}\label{gcomp}
\begin{cases}
g_{1}(s,\mu,\boldsymbol{\varPhi}^*,\varepsilon )
\boldsymbol{e}_{1}:=
P_{1}(F_{0}(s,\mu)+\varepsilon
    F_{1}(s,\mu,\boldsymbol{\varPhi}^*,\varepsilon)),\\
    g_{2}(s,\mu,\boldsymbol{\varPhi}^*,\varepsilon)\boldsymbol{e}_{2}:=
    P_{2}(F_{0}(s,\mu)+\varepsilon
    F_{1}(s,\mu,\boldsymbol{\varPhi}^*,\varepsilon)),\\
    g_{3}(s,\mu,\boldsymbol{\varPhi}^*,\varepsilon)
    :=\boldsymbol{\varPhi}^*+\mathcal{L}_{X_0}^{-1}(I-P)(F_{0}(s,\mu)+\varepsilon
    F_{1}(s, \mu, \boldsymbol{\varPhi}^*,\varepsilon)),
\end{cases}
\end{equation}
for any $(s,\mu, \boldsymbol{\varPhi}^*,\varepsilon)
\in \mathbb{R}^{2}\times X_0\times[0,\varepsilon_0]$ with small $\varepsilon_0>0$,
then \eqref{2.21} is reduced to the system
\begin{equation}\label{G=0}
    0=G(s,\mu,\boldsymbol{\varPhi}^*,\varepsilon):=
    \left[\begin{array}{c}
    g_{1}(s, \mu, \boldsymbol{\varPhi}^*,\varepsilon )\\g_{2}(s, \mu, \boldsymbol{\varPhi^*},\varepsilon )\\g_{3}(s,\mu,\boldsymbol{\varPhi}^*,\varepsilon )\\\end{array}\right].
\end{equation}

\subsection{The limiting system of \eqref{G=0} as $\varepsilon\to 0$ }
To construct the branch of solutions of
\eqref{G=0} with small $\varepsilon >0$,
we set $\varepsilon = 0$ in \eqref{gcomp} and \eqref{G=0} to
introduce the following limiting system
\begin{equation}\label{Gto0}
G(s,\mu,\boldsymbol{\varPhi}^*, 0 )=\left[
\begin{array}{l}
g_{1}^{0}(s,\mu)\\
g_{2}^{0}(s,\mu)\\
\boldsymbol{\varPhi}^*+
\mathcal{L}_{X_0}^{-1}(I-P)F_{0}(s,\mu)
\end{array}
\right]=
\left[
\begin{array}{l}
0\\
0\\
0
\end{array}
\right],
\end{equation}
where
$g_{1}^{0}(s,\mu)\boldsymbol{e}_{1}:=P_{1}F_{0}(s,\mu)$
and
$g_{2}^{0}(s,\mu)\boldsymbol{e}_{2}:=P_{2}F_{0}(s,\mu)$.

It is noted that $g^{0}_{1}(s, \mu )$ is independent of $s$, then
we denote $g^{0}_{1}(s,\mu )$ by $g^{0}_{1}(\mu )$.
Indeed,
from \eqref{st}, \eqref{P1P2} and \eqref{2.22},
one can see that
\begin{equation}\label{g1def}
\begin{split}
g_{1}^{0}(\mu)&=
\langle F_{0}(s,\mu),\boldsymbol{e}_{1}^{*}\rangle=
\dfrac{1}{|\Omega |}\displaystyle
\int_{\Omega }{f}^0_{1}(\mu)\,\dx\\[2mm]
&=\dfrac{a_1a_2}{b_{2}d_1|\Omega |}
\displaystyle\int_{\Omega}
\left(\frac{1}{1+\mu\varphi_{j}(x)}-\frac{B}{A}\frac{1}{(1+\mu\varphi_{j}(x))^2}\right)\dx,
\end{split}
\end{equation}
where we use the notation $A = \frac{a_{1}}{a_{2}}$ and $B= \frac{b_{1}}{b_{2}}$,
which make the subsequent calculations somewhat more transparent.
Moreover, our results concern the case $A > B$.
Similarly we obtain the expression of $g^{0}_{2}(s, \mu )$ as follows:
\begin{equation*}
\begin{split}
g_{2}^{0}(s,\mu)&=\langle F_{0}(s,\mu),\boldsymbol{e}_{2}^{*}\rangle
=-\biggl(\beta+\frac{b_{2}}{a_{2}}\biggr)
\displaystyle
\int_{\Omega }f^0_{1}(\mu)\varphi_j(x)\dx+\displaystyle
\int_{\Omega }f^0_{2}(s,\mu)\varphi_j(x)\dx\\
=&
-\biggl(\beta+\frac{b_{2}}{a_{2}}\biggr)
\frac{a_{2}}{b_{2}d_{1}}
\displaystyle\int_{\Omega}
\left[\frac{a_{1}\varphi_{j}(x)}{1+\mu\varphi_{j}(x)}
-\frac{a_{2}b_{1}\varphi_{j}(x)}{b_{2}(1+\mu\varphi_{j}(x))^{2}}\right]\dx\\
&+s\lambda^{2}_{j}\frac{b_{2}\mu}{a^{2}_{2}}+\frac{\lambda_{j}}{b_{2}}(a_{2}\beta+b_{2})
\displaystyle\int_{\Omega}\frac{\varphi_{j}(x)}{1+\mu\varphi_{j}(x)}\dx
\end{split}
\end{equation*}
\begin{equation}\label{g2def}
\begin{split}
=&s\mu\lambda^{2}_{j}\frac{b_{2}}{a^{2}_{2}}+
\frac{(a_{2}\beta+b_{2})a_1}{b_{2}\mu d_1}\left(-|\Omega|+
\dfrac{B}{A}
%\frac{a_2b_1}{a_1b_2}
\displaystyle\int_{\Omega}\frac{\dx}{1+\mu\varphi_{j}(x)}\right) \\ &+\frac{(a_{2}\beta+b_{2})}{b_{2}\mu}\lambda_j\left(|\Omega|-\displaystyle
\int_{\Omega}\frac{\dx}{1+\mu\varphi_{j}(x)}\right)\\
&+\frac{(a_{2}\beta+b_{2})a_1}{b_{2}\mu d_1}\left[\displaystyle\int_{\Omega}\frac{\dx}{1+\mu\varphi_{j}(x)}-
\dfrac{B}{A}
%\frac{b_1a_2}{b_2a_1}
\displaystyle\int_{\Omega}\frac{\dx}{(1+\mu\varphi_{j}(x))^2}
\right],
\end{split}
\end{equation}
where the last equality comes from
$$
\displaystyle\int_{\Omega}\dfrac{\varphi_{j}(x)}{1+\mu\varphi_{j}}\dx=
\dfrac{1}{\mu}\biggl(|\Omega |-
\displaystyle\int_{\Omega }\dfrac{\dx}{1+\mu\varphi_{j}(x)}\biggr),
$$
and
$$
\displaystyle\int_{\Omega }
\dfrac{\varphi_{j}(x)}{(1+\mu\varphi_{j})^{2}}\dx=
\dfrac{1}{\mu}\biggl\{
\displaystyle\int_{\Omega }\dfrac{\dx}{1+\mu\varphi_{j}(x)}-
\displaystyle\int_{\Omega }\dfrac{\dx}
{(1+\mu\varphi_{j}(x))^{2}}\biggr\}.
$$
Following an idea proposed by Lou, Ni and Yotsutani \cite{LNY2},
we introduce the function
\begin{equation}\label{21graph}
    h_{j}(\mu ):=\dfrac{\int_{\Omega}\frac{\dx}{(1+\mu\varphi_{j}(x))^{2}}}{\int_{\Omega}\frac{\dx}{1+\mu\varphi_{j}(x)}}
    \ \ \mbox{for}\ \ -\dfrac{1}{\max_{\overline{\Omega}}\varphi_{j}}<\mu<-\dfrac{1}{\min_{\overline{\Omega}}\varphi_{j
    }}.
\end{equation}
This function $h_{j}(\mu )$
will play an important role in constructing the set
of solutions of \eqref{Gto0}.
In
\cite{LNY2},
they obtained the following profile of $h_{j}(\mu )$
to construct a branch of solutions of \eqref{sSKT} perturbed by
that of solutions of the 1st shadow system \eqref{LNlim1}.
\begin{lem}[\cite{LNY2}]\label{stlem}
Assume that $\lambda_j>0$ is simple, let $m_{j}:=-\dfrac{1}{\max\limits_{\overline{\Omega}}\varphi_{j}}<0$
and $M_{j}:=-\dfrac{1}{\min\limits_{\overline{\Omega}}\varphi_{j}}>0$,
then $h_{j}(0)=1$, $h_j^{\prime}(0)=0$
and
$\mu h_{j}^{\prime}(\mu)>0$ for $\mu\in (m_{j},M_{j})\setminus
\{0\}$.
Furthermore if $N\le 4$, then
$\lim\limits_{\mu\downarrow m_{j}}h_{j}(\mu )=
\lim\limits_{\mu\uparrow M_{j}}h_{j}(\mu )=+\infty$.
\end{lem}
By virtue of Lemma \ref{stlem}, we have
\begin{prop}\label{prop1}
Suppose that $N\le 4$, $A>B$ and $\lambda_j>0$ is simple.
Then there exist
$\mu^{\pm}_{j,0}$ satisfying $m_{j}<\mu^{-}_{j,0}<0<\mu^{+}_{j,0}<M_{j}$ and
$$
h_j(\mu^{\pm}_{j,0})=\frac{A}{B}=\frac{a_1b_2}{a_2b_1}>1,
$$
thus $g^{0}_{1}( \mu^{\pm}_{j,0})=0$.
\end{prop}
Further we can prove the following inequalities:
\begin{prop}\label{prop2}
Suppose that $N\le 4$, $A>B$ and $\lambda_j>0$ is simple, and let $\mu^{\pm}_{j,0}$ be selected as in Proposition \ref{prop1}, i.e.
 $$
\displaystyle\int_{\Omega}\frac{\dx}{1+\mu^{\pm}_{j,0}\varphi_{j}(x)}=
\frac{B}{A}\displaystyle\int_{\Omega}\frac{\dx}{(1+\mu^{\pm}_{j,0}\varphi_{j}(x))^{2}},
$$
then
\begin{equation*}
\begin{split}
&(i)\ \ \displaystyle\int_{\Omega}\frac{\dx}{1+\mu^{\pm}_{j,0}\varphi_{j}(x)}>|\Omega|,\\
&(ii)\ \ \displaystyle\int_{\Omega}\frac{\dx}{1+\mu^{\pm}_{j,0}\varphi_{j}(x)}\le \frac{A}{B}|\Omega|,\\
&(iii)\ \ \frac{A}{B}\displaystyle\int_{\Omega}\frac{\dx}{(1+\mu^{\pm}_{j,0}\varphi_{j}(x))^{2}}
\leq \displaystyle\int_{\Omega}\frac{\dx}{(1+\mu^{\pm}_{j,0}\varphi_{j}(x))^{3}}.
\end{split}
\end{equation*}
\end{prop}

\begin{proof}
(i) Let $I(\mu):=\displaystyle\int_{\Omega}\frac{\dx}{1+\mu\varphi_{j}}$,
we observe that
$$
I'(\mu )=-\displaystyle\int_{\Omega }
\dfrac{\varphi_{j}}{(1+\mu\varphi_{j})^{2}}\dx,
\quad\mbox{especially,}\quad I'(0)=0,$$
and moreover,
$$
I''(\mu )=2\displaystyle\int_{\Omega }
\dfrac{\varphi_{j}^{2}}{(1+\mu\varphi_{j})^{3}}\dx>0
\quad\mbox{for any}\ \mu\in (m_{j}, M_{j}).
$$
Together with $I(0)=|\Omega |>0$,
we see that
$I(\mu )> I(0)=|\Omega |$ for any
$\mu\in (m_{j}, M_{j})$.\\

(ii) Note that
\begin{equation*}
\begin{split}
\int_{\Omega}\dfrac{\dx}{1+\mu^{\pm}_{j,0}\varphi_{j}}
\le& |\Omega |^{\frac{1}{2}}
\biggl(\int_{\Omega}\dfrac{\dx}{(1+\mu^{\pm}_{j,0}\varphi_{j})^{2}}\biggl)^{\frac{1}{2}}\\
=& \sqrt{\frac{A}{B}}|\Omega|^{\frac{1}{2}}
\biggl(\int_{\Omega}\dfrac{\dx}{1+\mu^{\pm}_{j,0}\varphi_{j}}\biggl)^{\frac{1}{2}},
\end{split}
\end{equation*}
which proves (ii).\\

(iii) Using the fact
\begin{equation*}
\int_{\Omega}\frac{\dx}{(1+\mu^{\pm}_{j,0}\varphi_{j})^{2}}
=\frac{A}{B} \int_{\Omega}\frac{\dx}{1+\mu^{\pm}_{j,0}\varphi_{j}},
\end{equation*}
then
\begin{equation*}
\begin{split}
\biggl(\int_{\Omega}\frac{\dx}{(1+\mu^{\pm}_{j,0}\varphi_{j})^{2}}\biggr)^{2}
&\le  \int_{\Omega}\frac{\dx}{1+\mu^{\pm}_{j,0}\varphi_{j}}
\int_{\Omega}\frac{\dx}{(1+\mu^{\pm}_{j,0}\varphi_{j})^{3}}\\
&=\frac{B}{A}\int_{\Omega}\frac{\dx}{(1+\mu^{\pm}_{j,0}\varphi_{j})^{2}}
\int_{\Omega}\frac{\dx}{(1+\mu^{\pm}_{j,0}\varphi_{j})^{3}};
\end{split}
\end{equation*}
one can prove $(iii).$
\end{proof}

\begin{lem}\label{limeqlem}
Suppose that $N\le 4$ and $A>B$, and let $\mu^{\pm}_{j,0}$ be selected as in Proposition \ref{prop1}, which satisfies  $m_{j}<\mu^{-}_{j,0}<0<\mu^{+}_{j,0}<M_{j}$ and
$$
h_j(\mu^{\pm}_{j,0})=\frac{A}{B}=\frac{a_1b_2}{a_2b_1}>1.
$$
Then the following properties hold:\\
(i) The equations $g^{0}_{1}(\mu)=g^{0}_{2}(s,\mu )=0$ admit roots $\mu=\mu^{\pm}_{j,0}$ and
\begin{equation}\label{sexp}
\begin{split}s=s^{\pm}_{j,0}:=
\dfrac{a_{1}a^{2}_{2}(a_2\beta+b_{2})}{b^{2}_{2}d_{1}(\mu^{\pm}_{j,0})^{2}\lambda^{2}_{j}}
\bigg(|\Omega|-\frac{B}{A}\int_{\Omega}\frac{\dx}{1+\mu^{\pm}_{j,0}\varphi_{j}(x)}\bigg)\\
+\dfrac{a^{2}_{2}(a_2\beta+b_{2})}{\lambda_{j}b^{2}_{2}(\mu^{\pm}_{j,0})^{2}}
\bigg(\int_{\Omega}\frac{\dx}{1+\mu^{\pm}_{j,0}\varphi_{j}(x)}-|\Omega|\bigg)>0.
\end{split}
\end{equation}
(ii) All solutions of \eqref{Gto0} are represented as
\begin{equation}\label{sollimeq}
(s,\mu, \boldsymbol{\varPhi}^*)=
(s^{\pm}_{j,0}, \mu^{\pm}_{j,0}, -\mathcal{L}_{X_0}^{-1}(I-P)F_{0}
(s^{\pm}_{j,0},\mu^{\pm}_{j,0})).
\end{equation}
\end{lem}
\begin{proof}
By \eqref{g1def} and Proposition \ref{prop1}, $g^{0}_{1}(\mu)=0$ has two roots $\mu=\mu^{\pm}_{j,0}$.
Setting $\mu=\mu^{\pm}_{j,0}$ in \eqref{g2def} , by simple computation we see that
$g_{2}^{0}(s,\mu^{\pm}_{j,0})=0$ holds if and only if
\begin{equation*}
\begin{split}
s=s^\pm_{j,0}=\frac{(a_{2}\beta+b_{2})a_1a_2^2}{d_1b^2_{2}\lambda^2_j(\mu^\pm_{j,0})^2} \left(|\Omega|-\frac{B}{A}\displaystyle\int_{\Omega}\frac{\dx}{1+ \mu^\pm_{j,0} \varphi_{j}(x)}\right) \\ +\frac{(a_{2}\beta+b_{2})a^2_2}{\lambda_jb^2_{2}(\mu^\pm_{j,0})^2}\left(\displaystyle\int_{\Omega}\frac{\dx}{1
+\mu^\pm_{j,0}\varphi_{j}(x)}-|\Omega|\right),
\end{split}
\end{equation*}
which with (i) and (ii) of Proposition \ref{prop2} implies that $s^{\pm}_{j,0}>0$.

Substituting $(s,\mu)=(s^{\pm}_{j,0}, \mu^{\pm}_{j,0})$
into the equations of $\boldsymbol{\varPhi}^*$ component of \eqref{Gto0}, we
obtain \eqref{sollimeq}.
\end{proof}

\subsection{Proofs of Theorems \ref{LSprop} and \ref{blthm}}
\begin{proof}[{\bf Proof of Theorem \ref{LSprop}}]
As a perturbation of solutions of the limiting system \eqref{Gto0},
we aim to construct the set of solutions of \eqref{G=0} for small $\varepsilon>0$, by applying the Implicit Function Theorem.
Note that  Lemma \ref{limeqlem} implies that
\begin{equation}\label{G0}
    G(s^{\pm}_{j,0}, \mu^{\pm}_{j,0},
    \boldsymbol{\varPhi}^{*\pm}_{j,0}, 0)=0,
\end{equation}
with $\boldsymbol{\varPhi}^{*\pm}_{j,0}
=\mathcal{L}_{X_0}^{-1}(I-P)F_{0}(s^{\pm}_{j,0}, \mu^{\pm}_{j,0})$.
To prove the existence of solutions in the form of \eqref{ppe} satisfying
$G(s,\mu,\boldsymbol{\varPhi}^*,\varepsilon)=0$ for small $\varepsilon>0$, it remains to
 check the following non-degeneracy of $G(s,\mu,\boldsymbol{\varPhi}^*,\varepsilon)=0$ at $(s^{\pm}_{j,0},\mu^{\pm}_{j,0},\boldsymbol{\varPhi}^{*\pm}_{j,0},0)$:
\begin{lem}\label{nondeglem}
Let $G_{(s,\mu,\boldsymbol{\varPhi}^*)}
(s^{\pm}_{j,0}, \mu^{\pm}_{j,0}, \boldsymbol{\varPhi}^{*\pm}_{j,0},0)
$
be the Fr\'echet derivative of
$G$ respect to
$(s,\mu,\boldsymbol{\varPhi}^*)$
around
$(s^{\pm}_{j,0},\mu^{\pm}_{j,0},\boldsymbol{\varPhi}^{*\pm}_{j,0},0)$.
Then $G_{(s,\mu,\boldsymbol{\varPhi}^*)}
(s^{\pm}_{j,0}, \mu^{\pm}_{j,0}, \boldsymbol{\varPhi}^{*\pm}_{j,0},0)\,:\,
\mathbb{R}^{2}\times X_0\to\mathbb{R}^{2}\times  X_0$ is invertible.
\end{lem}

\begin{proof}
The Fr\'echet derivative of $G$ can be expressed as
\begin{equation}\label{Fr}
    \begin{split}
       & G_{(s,\mu,\boldsymbol{\varPhi}^*)}
       (s^{\pm}_{j,0}, \mu^{\pm}_{j,0}, \boldsymbol{\varPhi}^{*\pm}_{j,0},0)\\
=&
\left[
\begin{array}{ccc}
0&(g_{1}^{0})_{\mu}( \mu^{\pm}_{j,0}) & 0\\
(g_{2}^{0})_{s}(s^{\pm}_{j,0}, \mu^{\pm}_{j,0}) &
(g_{2}^{0})_{\mu}(s^{\pm}_{j,0}, \mu^{\pm}_{j,0}) & 0\\
\mathcal{L}_{X_0}^{-1}(I-P)(F_{0})_{s}(s^{\pm}_{j,0}, \mu^{\pm}_{j,0}) &
\mathcal{L}_{X_0}^{-1}(I-P)(F_{0})_{\mu}(s^{\pm}_{j,0}, \mu^{\pm}_{j,0}) & I
\end{array}
\right].
    \end{split}
\end{equation}
Since $\mathcal{L}_{X_0}^{-1}(I-P)\,:\,
Y\to X_{0}$
%\mbox{Range}(\mathcal{L})\to X_0$
is bounded,
then
$\mathcal{L}_{X_0}^{-1}(I-P)(F_{0})_{s}(s^{\pm}_{j,0}, \mu^{\pm}_{j,0})$ and
$\mathcal{L}_{X_0}^{-1}(I-P)(F_{0})_{\mu}(s^{\pm}_{j,0}, \mu^{\pm}_{j,0})$
are also bounded by \eqref{2.22}.
Thus in order to prove that $G_{(s,t,\boldsymbol{\varPhi}^*)}
(s^{\pm}_{j,0}, \mu^{\pm}_{j,0}, \boldsymbol{\varPhi}^{*\pm}_{j,0},0)$
is invertible, we  only need to verify that
\begin{equation}\label{detnot0}
\mbox{Det}\biggl[
\begin{array}{cc}
0 & (g_{1}^{0})_{\mu}( \mu^{\pm}_{j,0})\\
(g_{2}^{0})_{s}(s^{\pm}_{j,0}, \mu^{\pm}_{j,0}) & (g_{2}^{0})_{\mu}(s^{\pm}_{j,0}, \mu^{\pm}_{j,0})
\end{array}
\biggr]=-(g_{1}^{0})_{\mu}(g_{2}^{0})_{s}(s^{\pm}_{j,0}, \mu^{\pm}_{j,0})
\neq 0.
\end{equation}
Differentiating  \eqref{g1def} and \eqref{g2def}, and then by Propositions \ref{prop1} and \ref{prop2}, we have
\begin{equation}\label{DET}
\begin{split}
&-(g_{2}^{0})_{s}(g_{1}^{0})_{\mu}(s^{\pm}_{j,0}, \mu^{\pm}_{j,0})\\
&=A\,\dfrac{\mu_{j,0}^{\pm}\lambda_{j}^{2}}{d_{1}|\Omega|}
\biggl[
\displaystyle\int_{\Omega}\dfrac{\varphi_{j}(x)}{(1+\mu^{\pm}_{j,0}\varphi_{j}(x))^{2}}\dx
-\dfrac{2B}{A}
\displaystyle\int_{\Omega}\dfrac{\varphi_{j}(x)}{(1+\mu^{\pm}_{j,0}\varphi_j(x))^{3}}\dx
\biggr]\\
&=A\,\dfrac{\lambda_{j}^{2}}{d_{1}|\Omega |}
\left[
\displaystyle\int_{\Omega}\dfrac{\dx}{(1+\mu^{\pm}_{j,0}\varphi_{j}(x))}\right.\\
&\left.-\left(1+\frac{2B}{A}\right)\displaystyle\int_{\Omega}\dfrac{\dx}{(1+\mu^{\pm}_{j,0}\varphi_{j}(x))^2}
+\frac{2B}{A}\displaystyle\int_{\Omega}\dfrac{\dx}{(1+\mu^{\pm}_{j,0}\varphi_{j}(x))^3}\right]\\
&=A\,\dfrac{\lambda_{j}^{2}}{d_{1}|\Omega |}
\biggl[-\left(1+\frac{B}{A}\right)
\displaystyle\int_{\Omega}\dfrac{\dx}{(1+\mu^{\pm}_{j,0}\varphi_{j}(x))^2}
+\frac{2B}{A}\displaystyle\int_{\Omega}\dfrac{\dx}{(1+\mu^{\pm}_{j,0}\varphi_{j}(x))^3}\biggr]\\
&\ge A\,\dfrac{\lambda_{j}^{2}}{d_{1}|\Omega |}
\left(1-\frac{B}{A}\right)\displaystyle\int_{\Omega}
\dfrac{\dx}{(1+\mu^{\pm}_{j,0}\varphi_{j}(x))^2}>0;
\end{split}
\end{equation}
which proves \eqref{detnot0} and  completes the proof of Lemma \ref{nondeglem}.
\end{proof}
Owing to Lemma \ref{nondeglem}, the application of the Implicit
Function Theorem for $G$ near $(s,\mu,\boldsymbol{\varPhi}^*,\varepsilon)=
(s^{\pm}_{j,0},\mu^{\pm}_{j,0},\boldsymbol{\varPhi}^{*\pm}_{j,0},0)$
ensures that,
for each $j$ and sign of $\pm$, there exists
a continuously differentiable function $(s^{\pm}_{j}(\varepsilon), \mu^{\pm}_{j}(\varepsilon ),
\boldsymbol{\varPhi}^{*\pm}_{j}(\varepsilon ))$
with respect to $\varepsilon\in [0,\overline{\varepsilon}]$
for some $\overline{\varepsilon}>0$ such that
\begin{equation}\label{impf}
\begin{cases}
G(s^{\pm}_{j}(\varepsilon ), \mu^{\pm}_{j}(\varepsilon ), \boldsymbol{\varPhi}^{*\pm}_{j}(\varepsilon ),\varepsilon)=0
\ \ \mbox{for any}\ \ \varepsilon\in [0,\overline{\varepsilon}],\\
(s^{\pm}_{j}(0), \mu^{\pm}_{j}(0), \boldsymbol{\varPhi}^{*\pm}_{j}(x,0))=
(s^{\pm}_{j,0}, \mu^{\pm}_{j,0}, \boldsymbol{\varPhi}^{*\pm}_{j,0}(x)).
\end{cases}
\end{equation}
Since the system $G(s,\mu,\boldsymbol{\varPhi}^*,\varepsilon)=0$
is equivalent to \eqref{2}, which completes the proof of Theorem \ref{LSprop}.
\end{proof}

\begin{proof}[{\bf Proof of Theorem \ref{blthm}}]
By \eqref{change}, substituting \eqref{ppe} into \eqref{2.5}, i.e.
$$
u^\pm_{j,\varepsilon}(x)=h_{1\varepsilon}(\phi^\pm_{j,\varepsilon}(x),\psi^\pm_{j,\varepsilon}(x)),\;\;
w^\pm_{j,\varepsilon}(x)=\frac{1}{\varepsilon}h_{2\varepsilon}(\phi^\pm_{j,\varepsilon}(x),\psi^\pm_{j, \varepsilon}(x)),
$$
we get  the branches  of nonconstant positive solutions of the 2nd shadow system
\eqref{LNlim2} for $\varepsilon\in (0,\overline{\varepsilon }]$; which are continuous differentiable in $\varepsilon$. By \eqref{2.5} \eqref{2.8} and \eqref{h}, it can be checked that
$$
u^\pm_{j,\varepsilon}(x)\to \frac{\phi^\pm_{j,0}(x)}{\psi^\pm_{j,0}(x)-\beta \phi^\pm_{j,0}(x)}=\frac{a_2}{b_2(1+\mu^\pm_{j,0}\varphi_{j}(x))}>0
\ {\rm as}\;\varepsilon \downarrow 0;
$$
and
$$
\varepsilon w^\pm_{j,\varepsilon}(x)\to  \psi^\pm_{j,0}(x)-\beta \phi^\pm_{j,0}(x)= \frac{b_2}{a_2}s^\pm_{j,0}(1+\mu^\pm_{j,0}\varphi_{j}(x))>0\
{\rm as}\;\varepsilon \downarrow 0;
$$
which completes the proof of Theorem \ref{blthm}.
\end{proof}

\section{Instability of steady states for the shadow system
near the blow-up point}
This section is devoted to a proof of Theorem \ref{instthm}.
Under the transformation \eqref{change}, the evolutional shadow system \eqref{2.2*} can be written as the evolutional
system \eqref{1}, i.e.
\begin{equation}\label{she}
\dfrac{\partial}{\partial t}
\biggl[
\begin{array}{c}
\frac{\varepsilon}{d_1}\,h_{1\varepsilon}(\phi(x,t), \psi(x,t))\\[2mm]
\frac{1}{d_2}\,h_{2\varepsilon}(\phi(x,t), \psi(x,t))
\end{array}
\biggr]
=
\mathcal{L}
\biggl[
\begin{array}{c}
\phi(x,t)\\
\psi(x,t)
\end{array}
\biggr]
+\varepsilon\,
F(\phi(x,t), \psi(x,t), \varepsilon),
\end{equation}
where $(h_{1\varepsilon}(\phi, \psi), h_{2\varepsilon}(\phi, \psi))$ and
$\mathcal{L}$ are defined by \eqref{2.5} and \eqref{Ldef},
respectively, and
\[F(\phi, \psi, \varepsilon ):=
\biggl[
\begin{array}{c}
f_{1}(\phi, \psi, \varepsilon )\\
f_{2}(\phi, \psi, \varepsilon )
\end{array}
\biggr]
\]
with $f_{1}$ and $f_{2}$  defined by \eqref{fgdef}, and
\begin{equation}\label{h12}
\begin{array}{l}
h_{1\varepsilon}(\phi,\psi)=h_{10}(\phi,\psi)+\varepsilon h^*_{1\varepsilon}(\phi,\psi),\;\;h_{10}(\phi,\psi)=\frac{\phi}{\psi-\beta \phi},\\
h_{2\varepsilon}(\phi,\psi)=h_{20}(\phi,\psi)+\varepsilon h^*_{2\varepsilon}(\phi,\psi),
\;\;h_{20}(\phi,\psi)=\psi-\beta \phi.
\end{array}
\end{equation}
We recall \eqref{4} to note
$$
F^0(\phi, \psi):=F(\phi, \psi, 0)=\left(
\begin{array}{cc}
\frac{1}{d_1}\frac{\phi}{\psi-\beta\phi}\left(a_1-b_1\frac{\phi}{\psi-\beta\phi}\right)\\[2mm]
\frac{\lambda^2_j}{a_2}(\psi-\beta \phi)-b_2\frac{\lambda^2_j}{a^2_2}\phi+\lambda_j\left(\beta+\frac{b_2}{a_2}\right)\frac{\phi}{\psi-\beta\phi}
\end{array}
\right).
$$
For any small $\varepsilon\in (0,\overline{\varepsilon}]$,
let $(\phi_{\varepsilon}(x), \psi_{\varepsilon}(x))$ be
a positive stationary solution of \eqref{she} obtained by
Theorem \ref{LSprop}. The linearized evolutional system \eqref{she} around $(\phi_{\varepsilon}(x), \psi_{\varepsilon}(x))$
is in the following form:
\begin{equation}\label{linearsys}
 T^{\varepsilon}(x)
\biggl[
\begin{array}{c}
\widehat{\phi}_t(x,t)\\
\widehat{\psi}_t(x,t)
\end{array}
\biggr]
=
\mathcal{L}
\biggl[
\begin{array}{c}
\widehat{\phi}(x,t)\\
\widehat{\psi}(x,t)
\end{array}
\biggr]
+\varepsilon F^{\varepsilon}_{(\phi, \psi)}
(\phi_{\varepsilon}(x), \psi_{\varepsilon}(x), \varepsilon )
\biggl[
\begin{array}{c}
\widehat{\phi}(x,t)\\
\widehat{\psi}(x,t)
\end{array}
\biggr];
\end{equation}
where
\begin{equation}\label{Mdef}
    T^{\varepsilon }(x):=\biggl[
    \begin{array}{cc}
    \frac{\varepsilon}{d_1}(h_{1\varepsilon})_\phi &
  \frac{\varepsilon}{d_1}(h_{1\varepsilon})_\psi  \\
    \frac{1}{d_2}(h_{2\varepsilon})_\phi & \frac{1}{d_2}(h_{2\varepsilon})_\psi    \end{array}\biggr]\biggl|_{(\phi,\psi)=(\phi_{\varepsilon}(x), \psi_{\varepsilon}(x))}
=T^0+\varepsilon T^\varepsilon_1(x),
\end{equation}
with
\begin{equation}\label{T0}
    T^{0}:=\biggl[
    \begin{array}{cc}
0 &0 \\
    -\frac{\beta \lambda_j}{a_2} & \frac{ \lambda_j}{a_2}  \end{array}\biggr].
\end{equation}
Here we note that $d_{2}=a_{2}/\lambda_{j}$ when $\varepsilon =0$.
In what follows, we denote
\begin{equation}\label{Jacobi}
F_{(\phi, \psi)}^{\varepsilon}(x)
:= F_{(\phi, \psi)}\bigl(\phi_{\varepsilon}(x),\,\psi_{\varepsilon}(x),\,\varepsilon \bigr),
\end{equation}
for the sake of simplicity of notation.
Denote
$$
(\phi_0(x),\psi_0(x))=(\phi^\pm _{j,0}(x),\psi^\pm _{j,0}(x))=
(s^\pm_{j,0},s^\pm_{j,0}(\beta+b_2/a_2(1+\mu^\pm_{j,0}\varphi_j(x))),
$$
\begin{equation}\label{F0}
    F^0_{(\phi, \psi)}(x)=\left(\begin{array}{cc}
c^0_{11}(x)&c^0_{12}(x)\\
c^0_{21}(x)&c^0_{22}(x)
\end{array}\right),
\end{equation}
and $l_0(x)=1+\mu^\pm_{j,0}\varphi_j(x)$, then
$$
\psi_0(x)-\beta\phi_0(x)=\dfrac{b_2}{a_2}s^\pm_{j,0}l_0(x),
\qquad
\frac{\phi_0(x)}{\psi_0(x)-\beta\phi_0(x)}=\frac{a_2}{b_2l_0(x)}.
$$
After some simple computation, it can be checked that
\begin{align*}
c^0_{11}(x)
&= \frac{1}{d_1}\frac{\psi_0(x)}{(\psi_0(x)-\beta\phi_0(x))^2}
   \left(a_1 - 2b_1\frac{\phi_0(x)}{\psi_0(x)-\beta\phi_0(x)}\right) \\[1mm]
&= \frac{1}{d_1}\left(\beta + \frac{b_2}{a_2}l_0(x)\right)
   \frac{a_1 a_2^2}{b_2^2 s^\pm_{j,0}}
   \left(\frac{1}{l_0(x)^2}-\frac{2B}{A\,l_0(x)^3}\right),
\\[2mm]
c^0_{12}(x)
&= -\frac{1}{d_1}\frac{\phi_0(x)}{(\psi_0(x)-\beta\phi_0(x))^2}
   \left(a_1 - 2b_1\frac{\phi_0(x)}{\psi_0(x)-\beta\phi_0(x)}\right) \\[1mm]
&= -\frac{1}{d_1}\frac{a_1 a_2^2}{b_2^2 s^\pm_{j,0}}
   \left(\frac{1}{l_0(x)^2}-\frac{2B}{A\,l_0(x)^3}\right),
\\[2mm]
c^0_{21}(x)
&= \frac{\lambda_j^2(b_2+a_2\beta)}{a_2^2}
   \left(-1 + \frac{a_2}{\lambda_j(\psi_0(x)-\beta\phi_0(x))}
   + \frac{a_2\beta\phi_0}{\lambda_j(\psi_0(x)-\beta\phi_0(x))^2}\right) \\[1mm]
&= -\frac{\lambda_j^2(b_2+a_2\beta)}{a_2^2}
   + \frac{a_2\lambda_j(b_2+a_2\beta)}{b_2^2 s^\pm_{j,0} l_0(x)^2}
     \left(\beta + \frac{b_2}{a_2}l_0(x)\right),
\\[2mm]
c^0_{22}(x)
&= \frac{\lambda_j^2}{a_2}
   - \frac{(b_2+a_2\beta)\lambda_j}{a_2}
     \frac{\phi_0(x)}{(\psi_0(x)-\beta\phi_0(x))^2} \\[1mm]
&= \frac{\lambda_j^2}{a_2}
   - \frac{a_2\lambda_j(b_2+a_2\beta)}{b_2^2 s^\pm_{j,0} l_0(x)^2}.
\end{align*}

In the following of this section, we investigate the existence and the precise location of an eigenvalue $\sigma$ near zero with an eigenfunction $(\widehat{\phi}(x),\widehat{\psi}(x))$ of the  following eigenvalue problem corresponding to  the  linearized system \eqref{linearsys},
 \begin{equation}\label{eige}
 \sigma T^{\varepsilon}(x)
\biggl[
\begin{array}{c}
\widehat{\phi}(x)\\
\widehat{\psi}(x)
\end{array}
\biggr]
=
\mathcal{L}
\biggl[
\begin{array}{c}
\widehat{\phi}(x)\\
\widehat{\psi}(x)
\end{array}
\biggr]
+\varepsilon
F^{\varepsilon}_{(\phi, \psi)}(x)
\biggl[
\begin{array}{c}
\widehat{\phi}(x)\\
\widehat{\psi}(x)
\end{array}
\biggr].
\end{equation}

\begin{thm}\label{thm5.1}
Assume that $\beta\ge 0$, $N\le 4$, $A>B$ and $\lambda_{j}$ is simple.
For any $j$ and each sign of $\pm$, let $\mu^{\pm}_{j,0}$ be selected as in Proposition \ref{prop1}, then
there exists $\overline{\varepsilon}>0$
such that for any $\varepsilon\in [0, \overline{\varepsilon}]$,
the eigenvalue problem \eqref{eige} has a
positive eigenvalue $\sigma_{\varepsilon }:=\sigma^{\pm}_{j}(\varepsilon )$
expressed as
\begin{equation}\label{eigenv}
    \sigma^{\pm}_{j}(\varepsilon )=
\varepsilon\varLambda^{\pm}_{j}(\varepsilon )
\end{equation}
with an eigenfunction
$(\widehat{\phi}_{\varepsilon }(x), \widehat{\psi}_{\varepsilon }(x))=
(\widehat{\phi}^{\pm}_{j}(x;\varepsilon ), \widehat{\psi}^{\pm}_{j}(x;\varepsilon ))$
expressed as
\begin{equation}\label{egeq}
    \biggl[\begin{array}{c}
    \widehat{\phi}^\pm_j(x,\varepsilon)\\
    \widehat{\psi}^\pm_j(x,\varepsilon)
    \end{array}
    \biggr]=\biggl[
    \begin{array}{c}
    1\\
    \beta+\frac{b_{2}}{a_2}\ell_0(x)
    \end{array}
    \biggr]
    +\gamma^{\pm}_{j}(\varepsilon )
    \biggl[
    \begin{array}{c}
    1\\
    \beta+\frac{b_{2}}{a_2}
    \end{array}
    \biggr]
    +\varepsilon\biggl[
    \begin{array}{c}
    \widetilde{\phi}^{\pm}_{j}(x;\varepsilon )\\
    \widetilde{\psi}^{\pm}_{j}(x;\varepsilon )
    \end{array}\biggr]   ,
\end{equation}
where
\begin{equation}\label{elldef}
\ell_0(x):= 1+\mu^{\pm}_{j,0}\varphi_{j}(x).
\end{equation}
Here $\varLambda^{\pm}_{j}(\varepsilon )$,
$\gamma^{\pm}_{j}(\varepsilon )\in\mathbb{R}$ and
$(\widetilde{\phi}^{\pm}_{j}(x;\varepsilon ),
    \widetilde{\psi}^{\pm}_{j}(x;\varepsilon ))\in X_0$
are continuous for $\varepsilon\in [0, \overline{\varepsilon}]$ and satisfy
\begin{equation}\label{eigenlim}
    \varLambda^{\pm}_{j}(0)=\lambda_{j},\quad
        \gamma^\pm_{j}(0) = 0,\quad
\biggl[\begin{array}{c}
\widetilde{\phi}_j(x,0)\\
\widetilde{\psi}_j(x,0)
\end{array}\biggr]
=
\dfrac{b_{2}\lambda_{j}^{2}}{a_{2}}\mathcal{L}_{X_0}^{-1}
\left[
\begin{array}{c}
0\\
1
\end{array}
\right].
\end{equation}
\end{thm}

\begin{proof}
Following the Lyapunov-Schmidt reduction procedure
in the previous section,
we seek for an eigenpair of \eqref{eige} in
the form
\begin{equation}\label{eigen1}
\sigma =\varepsilon\varLambda
\quad\mbox{and}\quad
 \biggl[\begin{array}{c}
    \widehat{\phi}(x)\\
    \widehat{\psi}(x)
    \end{array}
    \biggr]=\biggl[
    \begin{array}{c}
    1\\
    \beta+\frac{b_{2}}{a_2}\ell_0(x)
    \end{array}
    \biggr]
    +\gamma
    \biggl[
    \begin{array}{c}
    1\\
    \beta+\frac{b_{2}}{a_2}
    \end{array}
    \biggr]
    +\varepsilon\biggl[
    \begin{array}{c}
    \widetilde{\phi}(x)\\
    \widetilde{\psi}(x)
    \end{array}\biggr] ,
\end{equation}
where $\varLambda, \gamma\in \mathbb{R}$ and $(\widetilde{\phi}, \widetilde{\psi})\in
X_0=X\cap \mbox{Range}(\mathcal{L})$.

Substituting \eqref{eigen1} into \eqref{eige},
we see that the eigenvalue problem \eqref{eige}
is reduced to the system
\begin{equation}\label{egp2}
    \mathcal{L}\biggl[
    \begin{array}{c}
    \widetilde{\phi}\\
    \widetilde{\psi}
    \end{array}
    \biggr]
    =(\varLambda T^{\varepsilon}(x)-F^{\varepsilon}_{(\phi, \psi)}(x))
    \biggl[
    \begin{array}{l}
    1+\gamma +\varepsilon\widetilde{\phi}\\
    \beta (1+\gamma)+\frac{b_{2}}{a_2}(\ell_0(x)+\gamma )+\varepsilon\widetilde{\psi}
    \end{array}
    \biggr].
\end{equation}
It will be considered that
$(\varLambda, \gamma, \widetilde{\phi}, \widetilde{\psi})$
are the unknown variables for the system \eqref{egp2} with a given $\varepsilon>0$.
Our task is to construct solutions of \eqref{egp2}
by the Lyapunov-Schmidt reduction method for small enough $\varepsilon>0$.
To this end, we decompose \eqref{egp2} into
the $\mbox{Ker}(\mathcal{L})$ component and $\mbox{Range}(\mathcal{L})$ component
as follows:
\begin{equation}\label{LS2}
\begin{cases}
P_{i}(\varLambda T^{\varepsilon }(x)-F^{\varepsilon }_{(\phi, \psi )}(x))
\biggl(
\biggl[
    \begin{array}{l}
    1+\gamma \\
    \beta(1+\gamma)+\frac{b_{2}}{a_2}(\ell_0(x)+\gamma )
    \end{array}
    \biggr]
+
\varepsilon
    \biggl[
    \begin{array}{l}
    \widetilde{\phi}\\
    \widetilde{\psi}
    \end{array}
    \biggr]\biggr)
=\biggl[
\begin{array}{c}
0\\
0
\end{array}
\biggr],
\\[2mm]
    \biggl[\begin{array}{c}
         \widetilde{\phi} \\
         \widetilde{\psi}
    \end{array}\biggr]=
    \mathcal{L}_{X_0}^{-1}(I-P)\biggl((\varLambda T^{\varepsilon}(x)-F^{\varepsilon}_{(\phi, \psi)}(x))
\biggl[
\begin{array}{l}
    1+\gamma +\varepsilon\widetilde{\phi}\\
    \beta(1+\gamma)+\frac{b_{2}}{a_2}(\ell_0(x)+\gamma )+\varepsilon\widetilde{\psi}
    \end{array}
    \biggr]
    \biggr),
    \end{cases}
\end{equation}
with $P_{i}$ $(i=1,2)$   defined by \eqref{P1P2}.

Set
$$
Q^\varepsilon(\varLambda,\gamma):=(\Lambda  T^{\varepsilon}(x)-F^{\varepsilon}_{(\phi, \psi)}(x))
\biggl[
\begin{array}{l}
    1+\gamma \\
    \beta(1+\gamma)+\frac{b_{2}}{a_2}(\ell_0(x)+\gamma )
    \end{array}
    \biggr],
$$
and
$$
\widetilde{Q}^\varepsilon (\varLambda,\gamma, \widetilde{\phi},\widetilde{\psi}):
=(\varLambda T^{\varepsilon}(x)-F^{\varepsilon}_{(\phi, \psi)}(x))
\biggl[
\begin{array}{l}
    \widetilde{\phi} \\
    \widetilde{\psi}
    \end{array}
    \biggr],
$$
then \eqref{LS2} can be written as
\begin{equation}\label{LS22}
\begin{cases}
P_{1}(Q^{\varepsilon}(\varLambda,\gamma)+\varepsilon\widetilde{Q}^{\varepsilon}(\varLambda,\gamma, \widetilde{\phi},\widetilde{\psi}))=0,\\
P_{2}(Q^{\varepsilon}(\varLambda,\gamma)+\varepsilon\widetilde{Q}^{\varepsilon}(\varLambda,\gamma, \widetilde{\phi},\widetilde{\psi}))=0,
\\[2mm]
\biggl[
\begin{array}{l}
    \widetilde{\phi} \\
    \widetilde{\psi}
    \end{array}
    \biggr]=\mathcal{L}_{X_0}^{-1}(I-P)(Q^{\varepsilon}(\varLambda,\gamma)+\varepsilon\widetilde{Q}^{\varepsilon}(\varLambda,\gamma, \widetilde{\phi},\widetilde{\psi})).
\end{cases}
\end{equation}
Therefore, if we set
\begin{equation}\label{pij}
    \begin{array}{ll}
    p_{i1}(\varLambda, \gamma, \varepsilon)\boldsymbol{e}_{i}&:=P_{i}Q^{\varepsilon}(\varLambda,\gamma),\\
    p_{i2}(\varLambda, \gamma, \widetilde{\phi}, \widetilde{\psi}, \varepsilon)\boldsymbol{e}_{i}&:=P_{i}
    \widetilde{Q}^{\varepsilon}(\varLambda,\gamma, \widetilde{\phi},\widetilde{\psi}),
\end{array}
\end{equation}
for $i=1,2$,
then \eqref{LS22} is reduced to
\begin{equation}\label{LS22*}
    \begin{cases}
    p_{11}(\varLambda, \gamma, \varepsilon)+
    \varepsilon\,p_{12}(\varLambda, \gamma, \widetilde{\phi}, \widetilde{\psi}, \varepsilon)=0,\vspace{1mm}\\
    p_{21}(\varLambda, \gamma, \varepsilon)+
    \varepsilon\,p_{22}(\varLambda, \gamma, \widetilde{\phi}, \widetilde{\psi}, \varepsilon)=0,\vspace{1mm}\\
    \biggl[\begin{array}{c}
         \widetilde{\phi} \\
         \widetilde{\psi}
    \end{array}\biggr]=
    \mathcal{L}_{X_0}^{-1}(I-P)(Q^{\varepsilon}(\varLambda,\gamma)+\varepsilon\widetilde{Q}^{\varepsilon}(\varLambda,\gamma, \widetilde{\phi},\widetilde{\psi})).
\end{cases}
\end{equation}
After some simple  computation, it can be verified that
\begin{equation}\label{Q0}
\begin{array}{ll}
Q^0(\varLambda,\gamma)&:= (\varLambda T^{0}(x)-F^{0}_{(\phi, \psi)}(x))
\biggl[
\begin{array}{l}
    1+\gamma \\
    \beta(1+\gamma)+\dfrac{b_{2}}{a_2}(\ell_0(x)+\gamma )
    \end{array}
    \biggr]\\
&=
\left[
\begin{array}{l}
\displaystyle \frac{\gamma a_1 a_2^2}{d_1 b_2^2 s^\pm_{j,0}}
   \frac{1-l_0(x)}{l_0(x)^2}
   \left(1-\frac{2B}{A l_0(x)}\right) \\[3mm]
\displaystyle \frac{\lambda_j^2 b_2}{a_2^2}
   +\frac{\lambda_j b_2 (\varLambda -\lambda_j)}{a_2^2} l_0(x)
   +\frac{\gamma \Lambda \lambda_j b_2}{a_2^2}
   +\frac{\gamma \lambda_j(b_2+a_2\beta)}{b_2 s^\pm_{j,0}}
    \frac{1-l_0(x)}{l_0(x)^2}
\end{array}
\right]
\\[2mm]
&:=
\left[
\begin{array}{l}
q_{1}^{0}(x , \gamma)\\[1mm]
q_{2}^{0}(x , \varLambda, \gamma)
\end{array}
\right]
\end{array}
\end{equation}
To construct the solutions of
\eqref{LS22} for small $\varepsilon>0$,
we first focus on finding the solution $(\varLambda^0, \gamma^0)$ of  the first two equations of the  limiting system of \eqref{LS22*}  at $\varepsilon = 0$:
\begin{equation}\label{limeq1}
p_{11}(\varLambda^0, \gamma^0, 0)=0\ \ \mbox{and}\ \
p_{21}(\varLambda^0, \gamma^0, 0)=0.
\end{equation}
By \eqref{P1P2} \eqref{LS22} \eqref{pij} and \eqref{Q0}, we see that
\eqref{limeq1} is reduced to
\begin{equation}\label{limeq2}
\begin{cases}
0=p_{11}(\varLambda^0, \gamma^0, 0)=\dfrac{1}{|\Omega |}\displaystyle\int_\Omega q^0_1(x,\gamma^0)\dx,\vspace{1mm}\\
0=p_{21}(\varLambda^0,\gamma^0,0)=\displaystyle\int_\Omega \left[-\left(\beta +\frac{b_2}{a_2}\right)q_1^0(x,\gamma^0)+
q_{2}^{0}(x, \varLambda^0, \gamma^0)\right]\varphi_j(x)\dx,
\end{cases}
\end{equation}
with $q^0_1$ and $q^0_2$ are defined by \eqref{Q0}.

\begin{lem}\label{limlem}
Assume that $N\le 4$,
$A>B$ and $\lambda_j>0$ is simple.
Then the system \eqref{limeq2} has a unique solution $(\varLambda^0,\gamma^0)=(\lambda_{j},0)$.
\end{lem}
\begin{proof}
It follows from \eqref{Q0} and \eqref{limeq2} that
\begin{equation}\label{q10}
    \begin{split}
p_{11}(\varLambda, \gamma, 0)=&
\dfrac{1}{|\Omega |}\displaystyle\int_{\Omega}q^{0}_{1}(x,\gamma )\dx
=\dfrac{\gamma a^2_2a_1}{d_1b^2_{2}s^{\pm}_{j,0}|\Omega |}
\displaystyle\int_{\Omega }
\frac{1-l_0(x)}{l_0^2}\bigg(1-\frac{2B}{Al_0(x)}\bigg)\dx\\
=&
\dfrac{\gamma a^2_2a_1}{d_1b^2_{2}s^{\pm}_{j,0}|\Omega |}\displaystyle\int_{\Omega }
\left[-\frac{1}{l_0(x)}+\bigg(\frac{2B}{A}+1\bigg)\frac{1}{l^2_0(x)}-\frac{2B}{Al^3_0(x)}\right]\dx\\
=&
\dfrac{\gamma a^2_2a_1}{d_1b^2_{2}s^{\pm}_{j,0}|\Omega |}\displaystyle\int_{\Omega }
\left[\bigg(\frac{B}{A}+1\bigg)\frac{1}{l^2_0(x)}-\frac{2B}{Al^3_0(x)}\right]\dx:=C_0\gamma,
\end{split}
\end{equation}
where we used the fact that
$l_0(x)=1+\mu^\pm_{j,0}\varphi_j(x)$
satisfies
\[
\int_{\Omega}\dfrac{\dx}{l_{0}(x)}=
\dfrac{B}{A}\int_{\Omega}\dfrac{\dx}{l_{0}^{2}(x)}\]
by Proposition \ref{prop1}.
Furthermore,
by Proposition \ref{prop2}, we have
\begin{equation}\label{neko}
\displaystyle\int_{\Omega }
\left[\bigg(\frac{B}{A}+1\bigg)\frac{1}{l^2_0(x)}-\frac{2B}{A l^3_0(x)}\right]\dx
\le
\displaystyle\int_{\Omega} \bigg(\frac{B}{A}-1\bigg)\frac{1}{l^2_0(x)}\dx<0.
\end{equation}
Thus $C_0<0$ and $p_{11}(\varLambda^0, \gamma^0, 0)=0$ if and only if $\gamma^0=0$.

Since $\gamma^0=0$ and $q^{0}_{1}(x,0)\equiv 0$, the second equation of \eqref{limeq2} is
equivalent to
\begin{equation}\label{q20}
\begin{split}
0&=\displaystyle\int_{\Omega}q_{2}^{0}(x,
\varLambda^0, 0)\varphi_{j}(x)\,
\dx\\
&=\frac{b_{2}\lambda_{j}}{a_2^2}\displaystyle\int_{\Omega}\bigg[\lambda_j +(\varLambda^0-\lambda_{j})
(1+\mu^{\pm}_{j,0})\varphi_{j}(x)\bigg]\varphi_j(x)\, \dx\\
&=\frac{b_{2}\lambda_{j}}{a_2^2}(\varLambda^0-\lambda_{j})\mu^{\pm}_{j,0}\int_\Omega \varphi^2_j(x) \dx=\frac{b_{2}\lambda_{j}}{a_2^2}(\varLambda-\lambda_{j})\mu^{\pm}_{j,0}.
\end{split}
\end{equation}
Thus system \eqref{limeq2} has a unique solution  $(\varLambda^0,\gamma^0)=(\lambda_{j},0)$, which completes the proof of Lemma \ref{limlem} .
\end{proof}
We continue the proof of Theorem \ref{thm5.1}.
By Lemma \ref{limlem},
we see that
\begin{equation}
p_{11}(\lambda_{j}, 0, 0)=p_{21}(\lambda_{j}, 0, 0)=0.
\end{equation}
In view of the infinitely dimensional component $(\widetilde{\phi},\widetilde{\psi})$ of the limiting system of
\eqref{LS22*} at $\varepsilon=0$, we set $(\varLambda, \varepsilon )=(\lambda_{j}, 0)$
in the third equation to get
\begin{equation}\label{ppsi}
        \left[\begin{array}{c}
         \widetilde{\phi}_{0} \\
         \widetilde{\psi}_{0}
    \end{array}\right]:=
    \mathcal{L}_{X_0}^{-1}(I-P)
    \left[\begin{array}{c}q_1^0(x,0)\\q_2^0(x,\lambda_j,0)\end{array}\right]
    =\dfrac{b_{2}\lambda_{j}^{2}}{a^2_{2}}\mathcal{L}_{X_0}^{-1}(I-P)
    \left[
    \begin{array}{c}
    0\\
    1
    \end{array}
    \right].
\end{equation}
Since $\boldsymbol{e}_{1}^{*}=(1,0)$, $\boldsymbol{e}_{2}^{*}=(-\beta-\frac{b_2}{a_2},1)\varphi_j(x)\in\mbox{Ker}(\mathcal{L}^{*}))$ and
$\mbox{Range}(\mathcal{L})=(\mbox{Ker}(\mathcal{L}^{*}))^{\bot}$,
by the Fredholm alternative and the fact that $(0,1)\cdot\boldsymbol{e}^{*}_1=0$ and $(0,1)\cdot\boldsymbol{e}^{*}_2=0$ implies
$(0,1)\in Y_{0}$,
therefore, \eqref{ppsi} has a unique solution
$$\biggl[\begin{array}{c}
\widetilde{\phi}_{0}(x)\\
\widetilde{\phi}_{0}(x)
\end{array}\biggr]
=
\dfrac{b_{2}\lambda_{j}^{2}}{a^2_{2}}\mathcal{L}_{X_0}^{-1}
\left[
\begin{array}{c}
0\\
1
\end{array}
\right].
$$
Consequently, we see that \eqref{LS22*} has a solution expressed as
\eqref{eigenlim} for the limiting case $\varepsilon =0$.

In what follows, we construct the solutions of \eqref{LS22*}
with sufficiently small $\varepsilon>0$
as the perturbation of the limiting case.
To do so, we define an operator associated with \eqref{LS22*} as follows
\begin{equation}\label{Hdef}
\begin{split}
H(\varLambda, \gamma, \widetilde{\phi}, \widetilde{\psi}, \varepsilon )=
\left[
\begin{array}{l}
   p_{11}(\varLambda, \gamma, \varepsilon)+
    \varepsilon\,p_{12}(\lambda, \gamma, \widetilde{\phi}, \widetilde{\psi}, \varepsilon)\vspace{1mm}\\
    p_{21}(\varLambda, \gamma, \varepsilon)+
    \varepsilon\,p_{22}(\varLambda, \gamma, \widetilde{\phi}, \widetilde{\psi}, \varepsilon)\vspace{1mm}\\[2mm]
    \biggl[\begin{array}{c}
         \widetilde{\phi} \\
         \widetilde{\psi}
    \end{array}\biggr]-
    \mathcal{L}_{X_0}^{-1}(I-P)\biggl((\varLambda T^{\varepsilon}(x)-F^{\varepsilon}_{(\phi, \psi)}(x))\cdot\\[2mm]
    ~~~~~~~~\biggl[
    \begin{array}{l}
    (1+\gamma )+\varepsilon
    \widetilde{\phi}\\
    \beta(1+\gamma)+\frac{b_{2}}{a_2}(\ell_0(x)+\gamma )+\varepsilon
    \widetilde{\psi}
    \end{array}
    \biggr]\biggr)
\end{array}
\right].
\end{split}
\end{equation}
Hence the set of solutions of \eqref{LS22*} coincides with
the set of solutions of $H(\varLambda, \gamma, \widetilde{\phi}$, $\widetilde{\psi}, \varepsilon )=0$.
By the set \eqref{eigenlim} of solutions to \eqref{LS22*} with $\varepsilon =0$,
one can see
$H(\lambda_{j}, 0, \widetilde{\phi}_{0}, \widetilde{\psi}_{0}$,$0)=0$.
In order to apply the Implicit Function Theorem to
$H$ around
$(\lambda_{j}, 0, \widetilde{\phi}_{0}, \widetilde{\psi}_{0},  0)$,
we need to verify that
the Fr\'echet derivative
$H_{(\varLambda, \gamma, \widetilde{\phi}, \widetilde{\psi})}
(\lambda_{j}, 0$, $\widetilde{\phi}_{0}, \widetilde{\psi}_{0}, 0)\,:\,
\mathbb{R}\times\mathbb{R}\times X_0\to
\mathbb{R}\times\mathbb{R}\times X_0$
is invertible.
From \eqref{Hdef}, straightforward calculations yield
\begin{equation}\label{neko}
H_{(\varLambda, \gamma, \widetilde{\phi}, \widetilde{\psi})}
(\lambda_{j}, 0, \widetilde{\phi}_{0}, \widetilde{\psi}_{0}, 0)
=
\left[
\begin{array}{ccc}
(p_{11})_{\varLambda}(\lambda_{j}, 0, 0)
&
(p_{11})_{\gamma}(\lambda_{j}, 0, 0)
&
0
\\
(p_{21})_{\varLambda}(\lambda_{j}, 0, 0)
&
(p_{21})_{\gamma}(\lambda_{j}, 0, 0)
&
0
\\[2mm]
-\mathcal{L}_{X_0}^{-1}(I-P)Q^0_{\varLambda}(\lambda_{j}, 0)
&
-\mathcal{L}_{X_0}^{-1}(I-P)Q^0_{\gamma}(\lambda_{j}, 0)
&
I
\end{array}
\right],
\end{equation}
here $Q^0_{\varLambda}(\lambda_{j}, 0)=T^0(x)$ and
$Q^0_{\gamma}(\lambda_{j}, 0)=(\lambda_jT^0(x)-F^0_{(\phi,\psi)}(x))(1,\beta+\frac{b_2}{a_2})$ are bounded operators.
Then it suffices to show that
\begin{equation}\label{nonzero2}
\mbox{Det}\,\biggl[
\begin{array}{cc}
(p_{11})_{\varLambda}(\lambda_{j}, 0, 0)
&
(p_{11})_{\gamma}(\lambda_{j}, 0, 0)\\
(p_{21})_{\varLambda}(\lambda_{j}, 0, 0)
&
(p_{21})_{\gamma}(\lambda_{j}, 0, 0)
\end{array}
\biggr]
\neq 0.
\end{equation}
Note that
$$
(p_{11})_{\varLambda}(\lambda_{j}, 0, 0)=0,
$$
$$
(p_{11})_{\gamma}(\lambda_{j}, 0, 0)=\frac{a^2_2a_1}{d_1b^2_2s^\pm_{j,0}|\Omega|}\displaystyle\int_\Omega
\frac{l_0(x)-1}{l^2_0(x)}\left(1-\frac{2B}{Al_0(x)}\right)dx=C_0<0,
$$
by Lemma 5.2,
and
$$
(p_{21})_{\varLambda}(\lambda_{j}, 0, 0)=\frac{\lambda_jb_2}{a^2_2} \displaystyle\int_\Omega l_0(x)\varphi_j(x)dx
=\frac{\lambda_jb_2}{a^2_2}\mu^\pm_{j,0}\ne 0,
$$
thus
$$
\mbox{Det}\,\left[
\begin{array}{cc}
(p_{11})_{\varLambda}(\lambda_{j}, 0, 0)
&
(p_{11})_{\gamma}(\lambda_{j}, 0, 0)\\
(p_{21})_{\varLambda}(\lambda_{j}, 0, 0)
&
(p_{21})_{\gamma}(\lambda_{j}, 0, 0)
\end{array}
\right]
=-(p_{11})_{\gamma}
(p_{21})_{\varLambda}(\lambda_{j}, 0, 0)\ne 0.
$$
Then the Fr\'echet derivative
$H_{(\varLambda, \gamma, \widetilde{h}, \widetilde{k})}
(\lambda_{j}, 0,
\widetilde{h}_{0}, \widetilde{k}_{0}, 0):
\mathbb{R}\times\mathbb{R}\times X_0\to\mathbb{R}\times
\mathbb{R}\times X_0$
is invertible.
We apply the Implicit Function Theorem to $F$ around
$(\lambda_{j},0,\widetilde{h}_{0}, \widetilde{k}_{0}, 0)$
to get the local curve $(\varLambda_{\varepsilon}, \gamma_{\varepsilon }, \widetilde{\phi}_{\varepsilon},
\widetilde{\psi}_{\varepsilon })\in\mathbb{R}^{2}\times X_0$ for any $\varepsilon \in [0,\overline{\varepsilon}]$
with small $\overline{\varepsilon}>0$ satisfying \eqref{eigenlim}.
%$(\lambda_{0}, \gamma_{0})=(\lambda_{j},0)$.
Then the proof of Theorem \ref{thm5.1} is complete.
\end{proof}

\begin{proof}[Proof of Theorem \ref{instthm}]
By Theorem \ref{thm5.1},
the stationary solutions
$(\phi^{\pm}_{j,\varepsilon}(x), \psi^{\pm}_{j,\varepsilon}(x))$
of \eqref{she} are spectrally unstable if $\varepsilon\in
(0, \overline{\varepsilon}]$.
By virtue of the smoothness of the change of variables \eqref{change},
we see that the above spectral instability induces the spectral instability and nonlinear instability
of
$$
(u^{\pm}_{j,\varepsilon}(x), w^{\pm}_{j,\varepsilon}(x))=
\biggl(h_1(\phi^{\pm}_{j,\varepsilon}(x),\psi^{\pm}_{j,\varepsilon}(x) ,\frac{1}{\varepsilon}h_2(\phi^{\pm}_{j,\varepsilon}(x),\psi^{\pm}_{j,\varepsilon}(x))\biggr)
$$
to the evolutional system \eqref{2.2*}.
The proof of Theorem \ref{instthm} is complete.
\end{proof}

\section{Existence and instability  of coexistence states of the SKT model with large cross diffusion}
\subsection{Perturbation of positive stationary solutions of the 2nd shadow system}
In this section, we construct the branch of positive solutions to the original stationary SKT system \eqref{sSKT}
as the perturbation of the branch of solutions  of the 2nd shadow system \eqref{LNlim2}
when $\alpha$ is sufficiently large and $\varepsilon=a_2\lambda_j^{-1}-d_2$ is small.
In view of Theorem \ref{LNthm}(ii),
it is natural to set
$w=\alpha v$ in the  SKT model \eqref{sSKT}, which is a  perturbed system of the 2nd shadow system \eqref{LNlim2}.
Hence \eqref{sSKT}  is reduced to
\begin{equation}\label{sSKT0}
\begin{cases}
d_1\Delta[\,(1+w)u\,]+u(a_{1}-b_{1}u-\eta\, c_{1}w)=0\ \ &\mbox{in}\
\Omega,\\
d_{2}\Delta [(1+\beta u)w]+w(a_{2}-b_{2}u-\eta\, c_{2}w)=0\ \ &\mbox{in}\
\Omega,\\
\frac{\partial u}{\partial \nu }=\frac{\partial w}{\partial \nu }=0\ \ &\mbox{on}\
\partial\Omega,\\
u\ge 0,\ w\ge 0
\ \ &\mbox{in}\ \Omega,
\end{cases}
\end{equation}
with $\eta=1/\alpha$.
Our aim is to construct positive solutions of \eqref{sSKT0}
in the case when $\eta=1/\alpha$ and
$\varepsilon=a_2\lambda_j^{-1}-d_2>0$ are small enough.
For this end,
we use the change of variables \eqref{change} i.e.
$$
\tilde{w}=\varepsilon w,~~\phi=\varepsilon(1+w)u=(\varepsilon+\tilde{w})u,~
\psi=(1+\beta u)\tilde{w}
$$
to reduce
\eqref{sSKT0} to the following semilinear elliptic system:
\begin{equation}\label{nsemi2}
\begin{cases}
\Delta\phi +\varepsilon f_{1}(\phi, \psi, \varepsilon )-
\eta\, r_{1}(\phi, \psi, \varepsilon)
=0
\ \ &\mbox{in}\ \Omega,\\
\Delta\psi + \lambda_{j}\psi -\left(\beta+\frac{b_{2}\lambda_{j}}{a_{2}}\right)\phi
+\varepsilon f_{2}(\phi, \psi, \varepsilon )
-\eta\,
r_{2}(\phi, \psi, \varepsilon)
=0
\ \ &\mbox{in}\ \Omega,\\
\frac{\partial \phi}{\partial \nu }=\frac{\partial \psi}{\partial \nu }=0\ \ &\mbox{on}\
\partial\Omega,\\
\end{cases}
\end{equation}
where $f_1$ and $f_2$ are defined by \eqref{3},
$$
%\eta: =\dfrac{1}{\alpha},\quad
r_{1}(\phi, \psi, \varepsilon ):=
\dfrac{\varepsilon c_{1}h_{1\varepsilon}(\phi,\psi)h_{2\varepsilon}(\phi,\psi)}{d_{1}},\qquad
r_{2}(\phi, \psi, \varepsilon ):=
\dfrac{c_{2}h^2_{2\varepsilon}(\phi,\psi)\lambda_{j}}
{ \varepsilon (a_{2}-\varepsilon\lambda_{j})},
$$
with $h_{1\varepsilon}$ and $h_{2\varepsilon}$ defined by \eqref{2.5}.
%It should be noted that $r_{1}(\phi, \psi, \varepsilon)$ and
%$r_{2}(\phi, \psi, \varepsilon)$ remain uniformly bounded as
%$\varepsilon \to 0$.
Then \eqref{nsemi2} can be represented as the equation:
\begin{equation}\label{r12}
    \mathcal{L}\biggl[\begin{array}{c}
    \phi\\
    \psi\end{array}\biggr]+\varepsilon\biggl[\begin{array}{l}
    f_{1}(\phi,\psi,\varepsilon )\\
    f_{2}(\phi,\psi, \varepsilon )
    \end{array}\biggr]-\eta \biggl[\begin{array}{l}
    r_{1}(\phi,\psi,\varepsilon )\\
    r_{2}(\phi,\psi, \varepsilon )
    \end{array}\biggr]=
    \biggl[\begin{array}{c}
    0\\
    0\end{array}\biggr].
\end{equation}
Following the Lyapunov-Schmidt reduction framework in
the previous section,
we employ the decomposition \eqref{2.16}
for the unknown functions $(\phi, \psi)$
of \eqref{r12}, i.e.
$$
\biggl[\begin{array}{c}
    \phi(x)\\
    \psi(x)
    \end{array}\biggr]=s
    \biggl[\begin{array}{c}
    1\\
    \beta+\frac{b_{2}}{a_{2}}(1+\mu\varphi_{j}(x))
    \end{array}\biggr]+\varepsilon
    \biggl[\begin{array}{c}
    \phi^*(x)\\
    \psi^*(x)
    \end{array}\biggr].
$$
Then, \eqref{r12} can also be decomposed into three subspaces,
$\mathrm{Span}\{\boldsymbol{e}_{1}\}$,
$\mathrm{Span}\{\boldsymbol{e}_{2}\}$,
and $Y_{0}$, as follows:
\begin{equation}\label{K=0}
K(s,\mu,\boldsymbol{\varPhi}^*,\varepsilon,\eta)=0,
\end{equation}
with the operator $K\,:\,\mathbb{R}^{2}\times X_0\times \mathbb{R}^{2}\to\mathbb{R}^{2}\times X_0$ is defined by
\begin{equation}\label{Kdef}
    K(s,\mu,\boldsymbol{\varPhi}^*,\varepsilon,\eta):=G(s,\mu,\boldsymbol{\varPhi}^*,\varepsilon)-
    \dfrac{\eta}{\varepsilon}
    \left[\begin{array}{c}
    \kappa_{1}(s, \mu,\boldsymbol{\varPhi}^*,\varepsilon)\\\kappa_{2}(s,\mu,\boldsymbol{\varPhi}^*,\varepsilon)\\\kappa_{3}(s,\mu,\boldsymbol{\varPhi}^*,\varepsilon)\\\end{array}\right],
\end{equation}
where
$G(s,\mu,\boldsymbol{\varPhi}^*,\varepsilon)$
is defined by \eqref{gcomp} and \eqref{G=0},
and
$\kappa_{i}(s, \mu, \boldsymbol{\varPhi}^{*}, \varepsilon , \eta)$ are
defined by
\begin{equation}
       \kappa_{i}(s,\mu,\boldsymbol{\varPhi}^*, \varepsilon)\boldsymbol{e}_{i}=P_{i}R^*(s,\mu,\phi^*,\psi^*,\varepsilon )
        \quad\mbox{for}\ i=1,2
\end{equation}
with
\begin{equation}
   R^*(s,\mu,\phi^*,\psi^*,\varepsilon ):=R\left( s+\varepsilon\phi^*(x), s\left(\beta+\frac{b_{2}}{a_2}(1+\mu\varphi_{j}(x)
)\right)+\varepsilon\psi^*(x)),\varepsilon \right),
\end{equation}
and
$$    R(\phi, \psi, \varepsilon):=\biggl[
    \begin{array}{c}
    r_{1}(\phi, \psi, \varepsilon )\\
    r_{2}(\phi, \psi, \varepsilon )
    \end{array}
    \biggr],
$$
whereas $\kappa_{3}(s, \mu, \boldsymbol{\varPhi}^*, \varepsilon)$ is defined by
\begin{equation}
    \kappa_{3}(s, \mu, \boldsymbol{\varPhi}^*, \varepsilon, \eta ):=
    \mathcal{L}_{X_0}^{-1}(I-P)
    R^*(s,\mu,{\varPhi}^*, \varepsilon ).
    \nonumber
\end{equation}
\begin{lem}\label{nondeglem2}
Let $(s^{\pm}_{j}(\varepsilon ), \mu^{\pm}_{j}(\varepsilon ),
\boldsymbol{\varPhi}^{*\pm}_{j}(\varepsilon ))\in\mathbb{R}^2\times X_0$ be the
continuously differentiable function in \eqref{impf}.
Then there exists a small positive number $\overline{\varepsilon}$ such that,
for any $\underline{\varepsilon}\in (0,\overline{\varepsilon })$,
there exists $\delta=\delta(\underline{\varepsilon })>0$ such that
if $\varepsilon\in (\underline{\varepsilon }/2, 2\overline{\varepsilon })$ and
$\eta\in [0,\delta )$, then
$$
K_{(s,\mu, \boldsymbol{\varPhi}^*)} (s^{\pm}_{j}(\varepsilon ),
\mu^{\pm}_{j}(\varepsilon ),
\boldsymbol{\varPhi}^{*\pm}_{j}(\varepsilon ), \varepsilon, \eta )\,:\,
\mathbb{R}^{2}\times X_0
\to\mathbb{R}^{2}\times X_0
$$
is invertible.
\end{lem}

\begin{proof}
We recall that
$G_{(s,\mu,\boldsymbol{\varPhi}^*)}(s^{\pm}_{j,0}, \mu^{\pm}_{j,0},
\boldsymbol{\varPhi}^{\pm}_{j,0}, 0)\,:\,
\mathbb{R}^{2}\times X_0\to \mathbb{R}^{2}\times X_0$
is an isomorphism by Lemma \ref{nondeglem}.
By virtue of the perturbation property of the resolvent,
we can see that
$G_{(s,\mu,\boldsymbol{\varPhi}^*)}
(s^{\pm}_{j}(\varepsilon ), \mu^{\pm}_{j}(\varepsilon ),
\boldsymbol{\varPhi}^{\pm}_{j,0}, \varepsilon )$
is also an isomorphism if $\varepsilon \in
(0, \overline{\varepsilon })$
with some small $\overline{\varepsilon }>0$.
Therefore, in view of \eqref{Kdef},
for any $\underline{\varepsilon }\in (0,\overline{\varepsilon })$,
we can find a small $\delta =\delta (\underline{\varepsilon })>0$ such that
$K_{(s,\mu,\boldsymbol{\varPhi}^*)}
(s^{\pm}_{j}(\varepsilon ), \mu^{\pm}_{j}(\varepsilon ),
\boldsymbol{\varPhi}^{*\pm}_{j}(\varepsilon ), \varepsilon, \eta )$ is invertible
if $\varepsilon\in (\underline{\varepsilon}/2, 2\overline{\varepsilon})$
and $\eta\in (0,\delta)$.
\end{proof}
Let $\varepsilon\in (\underline{\varepsilon }, \overline{\varepsilon })$
be fixed arbitrary.
In view of \eqref{impf} and Lemma \ref{nondeglem2},
we know that
$$K(s^{\pm}_{j}(\varepsilon ), \mu^{\pm}_{j}(\varepsilon ),
\boldsymbol{\varPhi}^{*\pm}_{j}(\varepsilon ), \varepsilon, 0)=
G(s^{\pm}_{j}(\varepsilon ), \mu^{\pm}_{j}(\varepsilon ),
\boldsymbol{\varPhi}^{*\pm}_{j}(\varepsilon ), \varepsilon)=0$$
and $K_{(s,\mu,\boldsymbol{\varPhi}^*)}
(s^{\pm}_{j}(\varepsilon ), \mu^{\pm}_{j}(\varepsilon )
\boldsymbol{\varPhi}^{*\pm}_{j}(\varepsilon ), \varepsilon, 0)$
is invertible.
Therefore, the Implicit Function Theorem
yields a pair of neighborhoods
$\mathcal{U}^{\pm}$ of
$(s_{j}^{\pm}(\varepsilon ),
\mu_{j}^{\pm}(\varepsilon ),
\boldsymbol{\varPhi}^{*\pm}_{j}(\varepsilon ), \varepsilon, 0)$,
respectively,
small positive constants $\delta_1(\varepsilon)\in (0,\varepsilon)$
and $\delta_2(\underline{\varepsilon})$ such that
all the solutions in $\mathcal{U}^{\pm}$
of $K(s,\mu, \boldsymbol{\varPhi}^*, \varepsilon', \eta )=0$
can be represented as
\begin{equation}\label{curves}
    (s,\mu,\boldsymbol{\varPhi}^*)
=(s^{\pm}_{j}(\varepsilon', \eta), \mu^{\pm}_{j}(\varepsilon', \eta),
\boldsymbol{\varPhi}^{*\pm}_{j}(\varepsilon', \eta ))
\end{equation}
for any $(\varepsilon', \eta)\in (\varepsilon-\delta_1, \varepsilon+\delta_1)
\times (0, \delta_2)$.
Here
$(s^{\pm}_{j}(\varepsilon', \eta), \mu^{\pm}_{j}(\varepsilon', \eta),
\boldsymbol{\varPhi}^{*\pm}_{j}(\varepsilon', \eta ))$
is continuously differentiable in $\varepsilon'$ and $\eta$, and satisfies
$$(s^{\pm}_{j}(\varepsilon', 0), \mu^{\pm}_{j}(\varepsilon', 0),
\boldsymbol{\varPhi}^{*\pm}_{j}(\varepsilon', 0 ))=
(s^{\pm}_{j}(\varepsilon'), \mu^{\pm}_{j}(\varepsilon' ),
\boldsymbol{\varPhi}^{*\pm}_{j}(\varepsilon' )).$$
Thus we obtain the following existence result for the branch of
positive solutions of \eqref{nsemi2}:
\begin{thm}\label{thm6.1}
Assume that $N\le 4$, $A>B$
and $\lambda_{j}\in\mathbb{N}$ is simple.
For any such $j$ and each sign of $\pm$,
there exists $\overline{\varepsilon}>0$ such that
for any $\underline{\varepsilon}\in (0,\overline{\varepsilon })$,
there exists $\delta^{*}=\delta^{*}(\underline{\varepsilon })$ such that
if $\varepsilon\in (\underline{\varepsilon }, \overline{\varepsilon })$
and $\eta\in (0,\delta^{*})$,
\eqref{nsemi2}
has positive solutions
$
(\phi^{\pm}_{j}(x;\varepsilon, \eta ), \psi^{\pm}_{j}(x; \varepsilon, \eta ))
$
satisfying
\begin{equation}\label{ppe2}
    \biggl[
    \begin{array}{c}
    \phi^{\pm}_{j}(x; \varepsilon, \eta)\\
    \psi^{\pm}_{j}(x; \varepsilon, \eta)
    \end{array}
    \biggr]
    =s^{\pm}_{j}(\varepsilon , \eta )\biggl(
    \biggl[
    \begin{array}{c}
    1\\
    \beta+\frac{b_{2}}{a_2}
    \end{array}
    \biggr]
    +\mu^{\pm}_{j}(\varepsilon, \eta )
    \biggl[
    \begin{array}{c}
    0\\
    \frac{b_{2}}{a_2}\varphi_{j}(x)
    \end{array}
    \biggr]\biggr)
    +\varepsilon\biggl[
    \begin{array}{c}
    \widetilde{\phi}^{\pm}_{j}(x; \varepsilon, \eta)\\
    \widetilde{\psi}^{\pm}_{j}(x; \varepsilon, \eta)
    \end{array}
    \biggr],
\end{equation}
where $\boldsymbol{\varPhi}^{*\pm}_{j}(\varepsilon, \eta ):=
(\widetilde{\phi}^{\pm}_{j}(\varepsilon, \eta ),
\widetilde{\psi}^{\pm}_{j}(\varepsilon, \eta ))\in X_0$
and $(s^{\pm}_{j}(\varepsilon , \eta ), \mu^{\pm}_{j}(\varepsilon , \eta))\in \mathbb{R}^{2}$ are functions of $C^{1}$-class
for $(\varepsilon, \eta )\in (\underline{\varepsilon}, \overline{\varepsilon })\times [0, \delta^{*})$
satisfying
\begin{equation}
        (s^{\pm}_{j}(\varepsilon, 0),\mu^{\pm}_{j}(\varepsilon, 0))=(s^{\pm}_{j}(\varepsilon ), \mu^{\pm}_{j}(\varepsilon ))
        \ \ \mbox{and}\ \
        \boldsymbol{\varPhi}^{*\pm}_{j}(\varepsilon , 0)= \boldsymbol{\varPhi}^{*\pm}_{j}(\varepsilon ),
\nonumber
\end{equation}
where $(s^{\pm}_{j}(\varepsilon ), \mu^{\pm}_{j}(\varepsilon ), \boldsymbol{\varPhi}^{*\pm}_{j} (\varepsilon ))$
are defined by Theorem \ref{LSprop}.
\end{thm}

\subsection{Perturbation of unstable eigenvalues of the linearized problem
for the 2nd shadow system}
By setting $w=\alpha v$ in the original evolutional SKT  system \eqref{SKT},
we obtain
$$
\begin{cases}
u_{t}=\Delta[\,(d_{1}+w)u\,]+u(a_{1}-b_{1}u-\eta\,c_{1}w)\ \ &\mbox{in}\
\Omega\times (0,T),\\
w_{t}=d_{2}\Delta [(1+\beta u)w]+w(a_{2}-b_{2}u-\eta\,c_{2}w)\ \ &\mbox{in}\
\Omega\times (0,T),\\
\frac{\partial u}{\partial \nu }=\frac{\partial w}{\partial \nu }=0\ \ &\mbox{on}\
\partial\Omega\times(0,T),\\
\end{cases}
$$
where $\eta =1/\alpha $.
Furthermore, the change of variables \eqref{change} reduces the
above system to
\begin{equation}\label{evol}
\dfrac{\partial}{\partial t}
\biggl[
\begin{array}{c}
\frac{\varepsilon}{d_1}
h_{1\varepsilon}(\phi, \psi)\\
\frac{1}{d_2}
h_{2\varepsilon}(\phi, \psi)
\end{array}
\biggr]
=
\mathcal{L}
\biggl[
\begin{array}{c}
\phi\\
\psi
\end{array}
\biggr]
+
\varepsilon\biggl[
\begin{array}{c}
f_{1}(\phi, \psi, \varepsilon )\\
f_{2}(\phi, \psi, \varepsilon )
\end{array}
\biggr]
-\eta\biggl[
\begin{array}{c}
r_{1}(\phi, \psi, \varepsilon )\\
r_{2}(\phi, \psi, \varepsilon )
\end{array}
\biggr]
\end{equation}
In view of Theorem  \ref{thm6.1}, we denote
$$
(\phi_{\varepsilon, \eta}(x), \psi_{\varepsilon,  \eta}(x)):=
(\phi^{\pm}_{j}(x;\varepsilon, \eta), \psi^{\pm}_{j}(x;\varepsilon, \eta))
\quad\mbox{for}\quad
(\varepsilon, \eta )\in(\underline{\varepsilon}, \overline{\varepsilon})\times
[0, \delta^{*})
$$
by the stationary solutions of  system \eqref{evol}.
To investigate the  spectral  stability/instability of
$
(\phi_{\varepsilon, \eta}, \psi_{\varepsilon, \eta})$,
we consider the  linearized evolutional system of \eqref{evol} around $(\phi_{\varepsilon, \eta}, \psi_{\varepsilon, \eta})$, and the corresponding  eigenvalue problem of  the linearized evolutional system is as follows:
\begin{equation}\label{egeq2}
\sigma T^{\varepsilon, \eta}(x)\biggl[
\begin{array}{c}
\phi\\
\psi
\end{array}
\biggr]
=
\mathcal{L}
\biggl[
\begin{array}{c}
\phi\\
\psi
\end{array}
\biggr]
+(\varepsilon F^{\varepsilon, \eta}_{(\phi,\psi)}(x)
-\eta R^{\varepsilon, \eta}_{(\phi, \psi )}(x))
\biggl[
\begin{array}{c}
\phi\\
\psi
\end{array}
\biggr],
\end{equation}
where
$T^{\varepsilon, \eta}(x)$ and $F^{\varepsilon, \eta}_{(\phi, \psi)}(x)$
are respectively defined as \eqref{Mdef} and \eqref{Jacobi} with
$(\phi_{\varepsilon}, \psi_{\varepsilon })$ replaced by
$(\phi_{\varepsilon, \eta}, \psi_{\varepsilon, \eta})$.
As the $\eta$-perturbation of Theorem \ref{thm5.1},
we obtain the spectral instability of
$(\phi_{\varepsilon, \eta}, \psi_{\varepsilon, \eta})$ as follows:
\begin{thm}\label{thm6.2}
Assume that $N\le 4$, $A>B$ and $\lambda_{j}$
is simple.
For any such $j\ge 1$ and each sign of $\pm$,
the eigenvalue problem \eqref{egeq2} has a
positive eigenvalue
\begin{equation}\label{eigenv2}
    \sigma^{\pm}_{j}(\varepsilon, \eta)=
\varepsilon\varLambda^{\pm}_{j}(\varepsilon, \eta)
\end{equation}
with an eigenfunction
$(\widehat{\phi}_{\varepsilon, \eta}(x), \widehat{\psi}_{\varepsilon, \eta}(x)):=
(\widehat{\phi^{\pm}_{j}}(x; \varepsilon, \eta ), \widehat{\psi^{\pm}_{j}}(x; \varepsilon, \eta))$
expressed by
\begin{equation}\label{eigenf2}
    \biggl[\begin{array}{c}
    \widehat{\phi}_{\varepsilon, \eta}(x)\\
    \widehat{\psi}_{\varepsilon, \eta}(x)
    \end{array}
    \biggr]=\biggl[
    \begin{array}{c}
    1\\
    \beta+\frac{b_{2}}{a_2}\ell_0(x)
    \end{array}
    \biggr]
    +\gamma_{\varepsilon, \eta}
    \biggl[
    \begin{array}{c}
    1\\
    \beta+\frac{b_{2}}{a_2}
    \end{array}
    \biggr]
    +\varepsilon\biggl[
    \begin{array}{c}
    \widetilde{\phi}_{\varepsilon, \eta}(x)\\
    \widetilde{\psi}_{\varepsilon, \eta}(x)
    \end{array}\biggr]   ,
\end{equation}
where
\begin{equation}\label{elldef2}
\ell_0(x):=
    1+\mu^{\pm}_{j,0}\varphi_{j}(x).
\end{equation}
Here $\varLambda^\pm_j(\varepsilon, \eta)$,
$\gamma_{\varepsilon, \eta}\in\mathbb{R}$ and
$(\widetilde{\phi}_{\varepsilon, \eta},\widetilde{\psi}_{\varepsilon, \eta})\in Y_0$
are continuous for
$(\varepsilon, \eta)\in (\underline{\varepsilon}, \overline{\varepsilon})
\times [0,\delta^{*})$ and satisfy
\begin{equation}\label{eigenlim2}
    \varLambda_{\varepsilon, 0}=\varLambda^\pm(\varepsilon),\quad
        \gamma_{\varepsilon, 0} = \gamma^\pm(\varepsilon),
\quad (\widetilde{\phi}_{\varepsilon, 0}(x), \widetilde{\psi}_{\varepsilon, 0}(x))=
(\widetilde{\phi}^\pm (x,\varepsilon), \widetilde{\psi}^\pm(x,\varepsilon))
\end{equation}
with $(\varLambda^\pm(\varepsilon), \gamma^\pm(\varepsilon),
\widetilde{\phi}^\pm(x,\varepsilon), \widetilde{\psi}^\pm(x,\varepsilon ))$
obtained in Theorem \ref{thm5.1}.
\end{thm}

\begin{proof}
Following the idea of the proof of Theorem \ref{thm5.1},
in order
to seek for the eigenpairs of \eqref{egeq2} in the form
\begin{equation}
\sigma =\varepsilon\varLambda
\quad\mbox{and}\quad
 \biggl[\begin{array}{c}
    \phi(x)\\
    \psi(x)
    \end{array}
    \biggr]=\biggl[
    \begin{array}{c}
    1\\
    \beta+\frac{b_{2}}{a_2}\ell_0(x)
    \end{array}
    \biggr]
    +\gamma
    \biggl[
    \begin{array}{c}
    1\\
    \beta+\frac{b_{2}}{a_2}
    \end{array}
    \biggr]
    +\varepsilon\biggl[
    \begin{array}{c}
    \widetilde{\phi}(x)\\
    \widetilde{\psi}(x)
    \end{array}\biggr] ,
    \nonumber
\end{equation}
we substitute them into \eqref{egeq2} to get
the reduced system
\begin{equation}\label{egp3}
\begin{array}{c}
    \mathcal{L}\biggl[
    \begin{array}{c}
    \widetilde{\phi}\\
    \widetilde{\psi}
    \end{array}
    \biggr]
    =\\
    \biggl(\varLambda T^{\varepsilon, \eta}(x)-
F^{\varepsilon, \eta}_{(\phi, \psi)}(x)
+\dfrac{\eta}{\varepsilon}R^{\varepsilon, \eta}_{(\phi, \psi )}(x)\biggr)
    \biggl[
    \begin{array}{l}
    1+\gamma +\varepsilon
    \widetilde{\phi}(x)\\
    (1+\gamma)\left(\beta+\frac{b_{2}}{a_2}\right)+\frac{b_{2}}{a_2}\mu^\pm_{j,0}\varphi_j(x)+\varepsilon
    \widetilde{\psi}(x)
    \end{array}
    \biggr].
    \end{array}
\end{equation}
By applying the Lyapunov-Schmidt reduction argument
as in the proof of Theorem \ref{thm5.1},
we reduce
\eqref{egp3} to the system
\begin{equation}\label{Heq2}
\begin{array}{c}
H(\varLambda, \gamma, \widetilde{\phi}, \widetilde{\psi}, \varepsilon)\\
=-\dfrac{\eta}{\varepsilon} \left[
\begin{array}{l}
   \rho_{1}(\lambda, \gamma, \widetilde{\phi}, \widetilde{\psi}, \varepsilon)\vspace{1mm}\\
   \rho_{2}(\lambda, \gamma, \widetilde{\phi}, \widetilde{\psi}, \varepsilon)\vspace{1mm}\\
    -
    \mathcal{L}_{X_0}^{-1}(I-P)R^{\varepsilon}_{(\phi, \psi)}(x)
    \biggl[
    \begin{array}{l}
    1+\gamma +\varepsilon
    \widetilde{\phi}\\
     \beta+\frac{b_{2}}{a_2}\ell_0(x)+\gamma(\beta+\frac{b_{2}}{a_2})+\varepsilon
    \widetilde{\psi}
    \end{array}
    \biggr]
\end{array}
\right]
\end{array}
\end{equation}
where
$H$ is the operator defined by \eqref{Hdef} and
$$
\rho_{i}(\lambda, \gamma, \widetilde{\phi}, \widetilde{\psi}, \varepsilon)
\boldsymbol{e}_{i}:=P_{i}
\biggl(
R^{\varepsilon, \eta }_{(\phi, \psi )}(x)
\biggl[
    \begin{array}{l}
    1+\gamma +\varepsilon\widetilde{\phi}\\
    (1+\gamma)\left(\beta+\frac{b_{2}}{a_2}\right)+\frac{b_{2}}{a_2}\mu^\pm_{j,0}\varphi_j(x)+\varepsilon
    \widetilde{\psi}
    \end{array}
    \biggr]\biggr).
$$
By \eqref{nonzero2}, we recall that
$H_{(\Lambda, \gamma, \widetilde{\phi}, \widetilde{\psi})}
(\lambda_{j}, 0, \widetilde{\phi}_{0}, \widetilde{\psi}_{0}, 0)\,:\,
\mathbb{R}\times\mathbb{R}\times X_0\to
\mathbb{R}\times\mathbb{R}\times Y_0$
is invertible.
By the perturbation of the resolvent,
one can see that
$H_{(\varLambda, \gamma, \widetilde{\phi}, \widetilde{\psi})}
(\lambda_{\varepsilon}, \gamma_{\varepsilon},
\widetilde{h}_{\varepsilon}$, $\widetilde{k}_{\varepsilon}, \varepsilon)$
is also invertible if $\varepsilon\in (0,\overline{\varepsilon })$
with some $\overline{\varepsilon }>0$.
In view of the left-hand side
of \eqref{Heq2},
one can see that, for any
$\underline{\varepsilon }\in (0, \overline{\varepsilon })$,
there exists $\delta^{*}=\delta^{*}(\underline{\varepsilon })>0$
such that if $\varepsilon\in (\underline{\varepsilon}, \overline{\varepsilon})$
and $\eta\in [0,\delta^{*})$,
the application of the Implicit Function Theorem near
$$(\varLambda, \gamma, \widetilde{\phi}, \widetilde{\psi}, \varepsilon, \eta)=
(\varLambda^\pm(\varepsilon), \gamma^\pm(\varepsilon),
\widetilde{\phi}^\pm(x,\varepsilon), \widetilde{\psi}^\pm(x,\varepsilon), \varepsilon, 0)
$$
enables us to construct the set of solutions of
\eqref{Heq2} expressed as \eqref{eigenf2} and \eqref{eigenlim2}.
Then the proof of Theorem \ref{thm6.2} is complete.
\end{proof}

\begin{proof}[Proof of Theorem \ref{SKTthm}]
By substituting \eqref{ppe2} into \eqref{change} with $w=\alpha v$, then by \eqref{2.5}
we see that
\begin{equation}\label{change2}
\begin{array}{l}
u^\pm_{j,\varepsilon ,\alpha}(x):=
h_{1\varepsilon}(\phi^{\pm}_{j}(x,\varepsilon,\alpha^{-1}),\psi^{\pm}_{j}(x;\varepsilon,\alpha^{-1})),\\
v^\pm_{j,\varepsilon ,\alpha}(x):=
\dfrac{h_{2\varepsilon}(\phi^{\pm}_{j}(x,\varepsilon,\alpha^{-1}),\psi^{\pm}_{j}(x;\varepsilon,\alpha^{-1}))}
{\alpha\varepsilon },
\quad
d_{2}=\dfrac{a_{2}}{\lambda_{j}}-\varepsilon,
\end{array}
\end{equation}
 when
$\varepsilon\in (\underline{\varepsilon}_{j,\pm},
\overline{\varepsilon}_{j,\pm})$ and $\alpha>0$ is sufficiently large.
Hence \eqref{change2} transforms the curve \eqref{ppe2} of
solutions of \eqref{nsemi2} to $\varGamma^{(\alpha )}_{j,\pm}$
stated in Theorem \ref{SKTthm}.
Furthermore, \eqref{uni} follows from the regularity of \eqref{change2} and Theorem \ref{thm6.1}.

The spectral instability of solutions of $\varGamma^{(\alpha )}_{j,\pm}$
follows from the spectral instability of solutions on the curve \eqref{ppe2}
and the regularity of \eqref{change2}, which  completes the proof of Theorem \ref{SKTthm}.
\end{proof}

\section*{Acknowledgment}
%The authors would like to thank the referees for their many valuable comments
%and useful suggestions which helped improve the exposition
%of the current paper.
Kousuke Kuto  was supported by
was supported by
JSPS KAKENHI Grant-in-Aid for Scientific Research (B),
Grant Number 25K00917, Yaping Wu was supported by NSF of China (No. 12371209 and No. 11871048) and Beijing NSF (No. 1232004).

%%%%%%%%%%%%%%%%%%%%%%%%%%%%%%%%%%%%%%%%%%%%%%%%%%%%%%%%%%%%%%%%%%%%%%%%%%%

\end{document}